\documentclass{amsart}
\usepackage{listings}
\usepackage{graphicx}
\newtheorem{theorem}{Theorem}[section]

\def\B#1#2{{#1\choose #2}}

\title{On the Dimension and Euler characteristic of random graphs}
\date{December 23, 2011}
\author{Oliver Knill}
\email{knill@math.harvard.edu}
\address{
        Department of Mathematics \\
        Harvard University \\
        Cambridge, MA, 02138
        }
\subjclass{Primary: 05C80,05C82,05C10,90B15,57M15 }
\keywords{Random Graph Theory, Complex Networks, Graph Dimension, Graph Curvature, Euler characteristic}

\begin{document}
\maketitle

\begin{abstract}
The inductive dimension ${\rm dim}(G)$ of a finite undirected graph $G$ is a rational number
defined inductively as $1$ plus the arithmetic mean of the dimensions of the unit spheres 
${\rm dim}(S(x))$ at vertices $x$ primed by the requirement that the empty graph has dimension $-1$.
We look at the distribution of the random variable ${\rm dim}$ on the Erd\"os-R\'enyi probability space $G(n,p)$,
where each of the $n(n-1)/2$ edges appears independently with probability $0\leq p \leq 1$.
We show that the average dimension $d_n(p) = {\rm E}_{p,n}[{\rm dim}]$
is a computable polynomial of degree $n(n-1)/2$ in $p$. The explicit formulas allow 
experimentally to explore limiting laws for the dimension of large graphs. In this context of random 
graph geometry, we mention explicit formulas for the expectation ${\rm E}_{p,n}[{\chi}]$ of the 
Euler characteristic $\chi$, considered as a random variable on $G(n,p)$. We look experimentally at
the statistics of curvature $K(v)$ and local dimension ${\rm dim}(v) = 1+{\rm dim}(S(v))$ 
which satisfy $\chi(G) = \sum_{v \in V} K(v)$ and ${\rm dim}(G) = \frac{1}{|V|}  \sum_{v \in V} {\rm dim}(v)$.
We also look at the signature functions
$f(p)={\rm E}_{p}[{\rm dim}], g(p)={\rm E}_p[\chi]$ and matrix values functions
$A_{v,w}(p) = {\rm Cov}_p[{\rm dim}(v),{\rm dim}(w)], B_{v,w}(p) = {\rm Cov}[K(v),K(w)]$ on the 
probability space $G(p)$ of all subgraphs of a host graph $G=(V,E)$ with the same vertex set $V$, 
where each edge is turned on with probability $p$. 
\end{abstract}

\section{Dimension and Euler characteristic of graphs} 

The inductive dimension for graphs $G=(V,E)$ is formally close to the Menger-Uhryson 
dimension in topology. It was in \cite{elemente11} defined as 
$$ {\rm dim}(\emptyset)= -1, {\rm dim}(G) = 1+\frac{1}{|V|} \sum_{v \in V} {\rm dim}(S(v))  \; ,  $$
where $S(v)=\{ w \in V \; | \; (w,v) \in E \; \}, \{ e=(a,b) \in E \; | \; (v,a) \in E, (v,b) \in E \; \}$ 
denotes the unit sphere of a vertex $v \in V$. 
The inductive dimension is useful when studying the geometry of graphs. We can look at the 
local dimension ${\rm dim}(v)=1+{\rm dim}(S(v))$ of a vertex which is a local property
like "degree" ${\rm deg}(v)={\rm ord}(S(v))$, or "curvature" $K(v)$ defined below. 
Dimension is a rational number defined for every finite graph; however it is in general not an integer. \\

The dimension is zero for graphs $P_n$ of size $0$, graphs which are completely disconnected. 
It is equal to $n-1$ for complete graphs $K_n$ of order $n$, where the size 
is $\B{n}{2} = n(n-1)/2$. Platonic solids can have dimension $1$ like for the cube and dodecahedron, it can be
$2$ like for the octahedron and icosahedron or be equal to $3$ like for the tetrahedron. The 600 cell with 
120 vertices is an example of a three dimensional graph, where each unit sphere is a two dimensional icosahedron.
Figure~\ref{archimedean} illustrates two Archimedean solids for which fractional dimensions occur 
in familiar situations. All Platonic, Archimedean and Catalan solids are graph theoretical polyhedra: 
a finite truncation or kising process produces two-dimensional graphs. \\

Computing the dimension of the unit spheres requires to find all unit spheres for vertices in the unit 
sphere $S(x)$ of a vertex and so on. The local dimension ${\rm dim}(x) = 1+{\rm dim}(S(x))$ satisfies by definition
\begin{equation}
 {\rm dim}(G) = \frac{1}{|V|} \sum_{x \in V} {\rm dim}(x)  \; . 
\end{equation}
There is an upper bound of the dimension in terms of
the average degree ${\rm deg}(G) = \frac{1}{|V|} \sum_V {\rm deg}(v)$ of a graph:
${\rm dim}(G) \leq {\rm deg}(G)-1$ with
equality for complete graphs $K_n$. The case of trees shows
that ${\rm dim}(G)=1$ is possible for arbitrary large $|V|$ or ${\rm deg}(G)$. \\

\begin{figure}
\scalebox{0.34}{\includegraphics{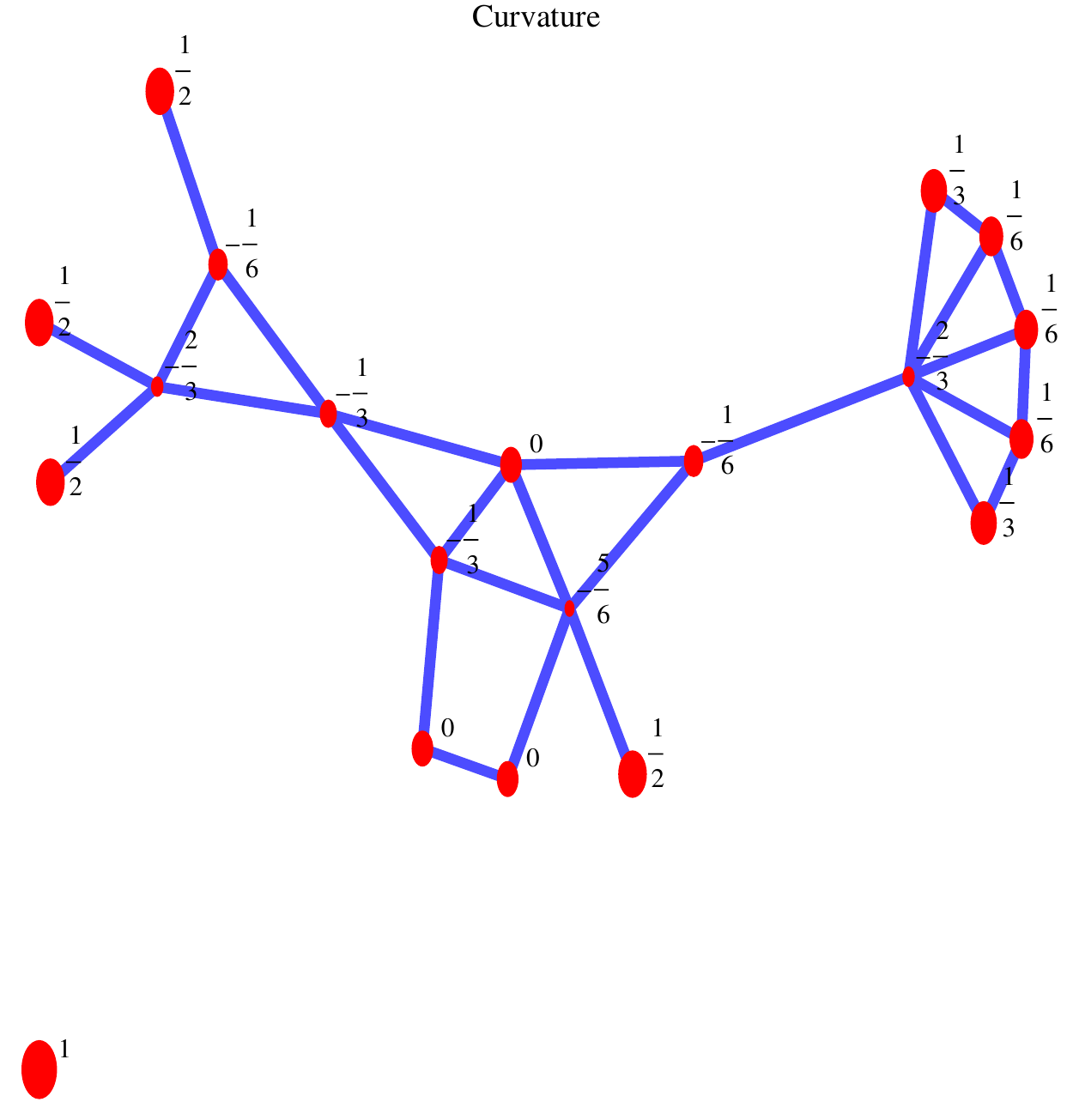}} 
\scalebox{0.34}{\includegraphics{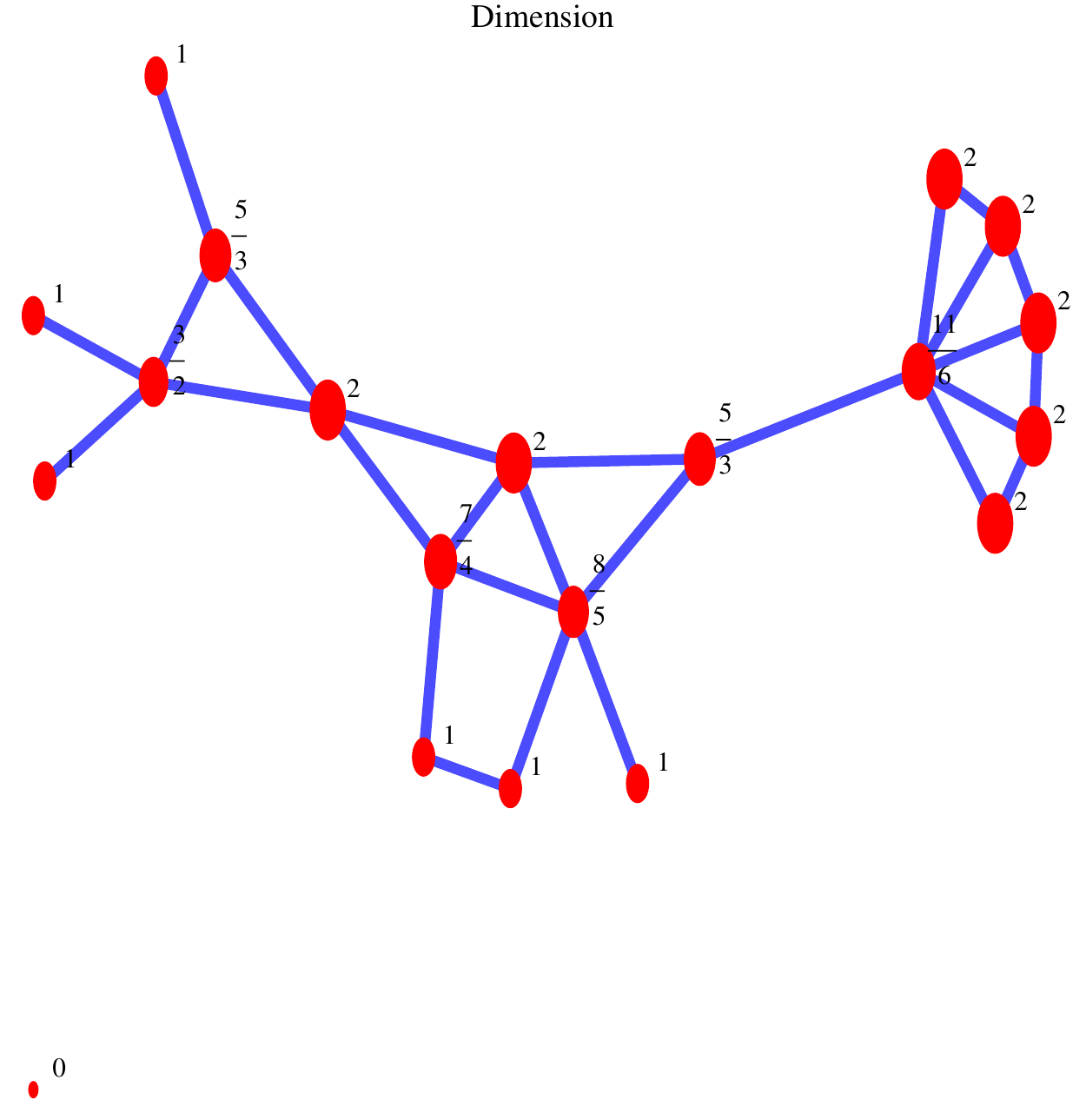}} 
\caption{
This figure shows a particular graph $G$ with $20$ vertices
and Euler characteristic $\chi(G) = 1$. 
The left side shows the graph with curvatures
$$\{{0,-1,-2,6,3,-5,-2,-4,3,0,-4,-1,0,2,1,1,1,2,3,3}\}/6$$
which sum up to the Euler characteristic $1$ by Gauss-Bonnet. 
The right figure shows the same graph with local dimensions
$$\{2,5/3,2,0,1,8/5,7/4,3/2,1,1,11/6,5/3,1,2,2,2,2,2,1,1\;\} $$
which average to the dimension of the graph ${\rm dim}(G)=1801/1200$.
}
\label{illustration}
\end{figure}

An other natural quantity for graphs is the Euler characteristic 
\begin{equation}
 \chi(G) = \sum_{k=0}^{\infty} (-1)^k v_k \; , 
 \label{eulercharacteristicdef}
\end{equation}
where $v_k$ is the number of $K_{k+1}$ subgraphs in $G$. We noted in \cite{cherngaussbonnet} that
it can be expressed as the sum over all curvature
$$  K(p) = \sum_{k=0}^{\infty} (-1)^k \frac{V_{k-1}(x)}{k+1} \; , $$
where $V_k(x)$ is the number of $K_{k+1}$ subgraphs in the sphere $S(x)$ at a vertex $x$. 
As in the continuum, the Gauss-Bonnet formula
\begin{equation}
 \label{gaussbonnet}
 \chi(G) = \sum_{x \in V} K(x) 
\end{equation}
relates a local quantity curvature with the global topological invariant $\chi$. 
For example, for a graph without triangles and especially for trees, 
the curvature is $K(v) = 1-{\rm deg}(v)/2$. For graphs without $4$ cliques
and especially two dimensional graphs $K(v) = 1-{\rm deg}(v)/2 - {\rm size}(S(v))/3$. 
For geometric graphs, where each sphere $S(v)$ is a cyclic graph, $K(v) = 1-{\rm deg}(v)/6$.
For the standard Petersen graph $P_{5,2}$, the dimension is $1$, the 
local dimension constant $1$, the Euler characteristic $-5$ and the curvature is constant $-1/2$ at every vertex. 
The Petersen graph $P_{9,3}$ has dimension $4/3$ and Euler characteristic $-6$. There are $9$ 
vertices with curvature $-1/2$ and $9$ with curvature $-1/6$. The sum of curvatures is $-6$.  \\

The Gauss-Bonnet relation (\ref{gaussbonnet}) is already useful for computing the Euler characteristic. 
For inhomogeneous large graphs especially, the Gauss-Bonnet-Chern formula simplifies in an elegant way 
the search for cliques in large graphs. An other application is the study of higher dimensional polytopes. 
It follows immediately for example that there is no $4$-dimensional polytope
- they are usually realized as a convex set in $R^5$ - for which the graph theoretical unit sphere is 
a three dimensional 600 cell: the Euler characteristic is $2$ in that dimension
and the curvature would by regularity have to be $2/|V|$. But curvature of such a graph would be constant
and since the 600 cell has 120 vertices, 720 edges, 1200 faces and 600 chambers, the curvature  were constant
$$ K = \frac{V_{-1}}{1} - \frac{V_0}{2} + \frac{V_1}{3} - \frac{V_2}{4} +\frac{V_3}{5} 
     = \frac{1}{1} - \frac{120}{2} + \frac{720}{3} - \frac{1200}{4} + \frac{600}{5} = 1 $$
which would force $|V|=2$ and obviously does not work. While such a result could certainly also 
be derived also with tools developed by geometers like Schl\"afli or Coxeter, the just graph 
theoretical argument is more beautiful. It especially does not need any ambient space realization
of the polytope. \\

\begin{figure}
\parbox{6.0cm}{ \scalebox{0.22}{\includegraphics{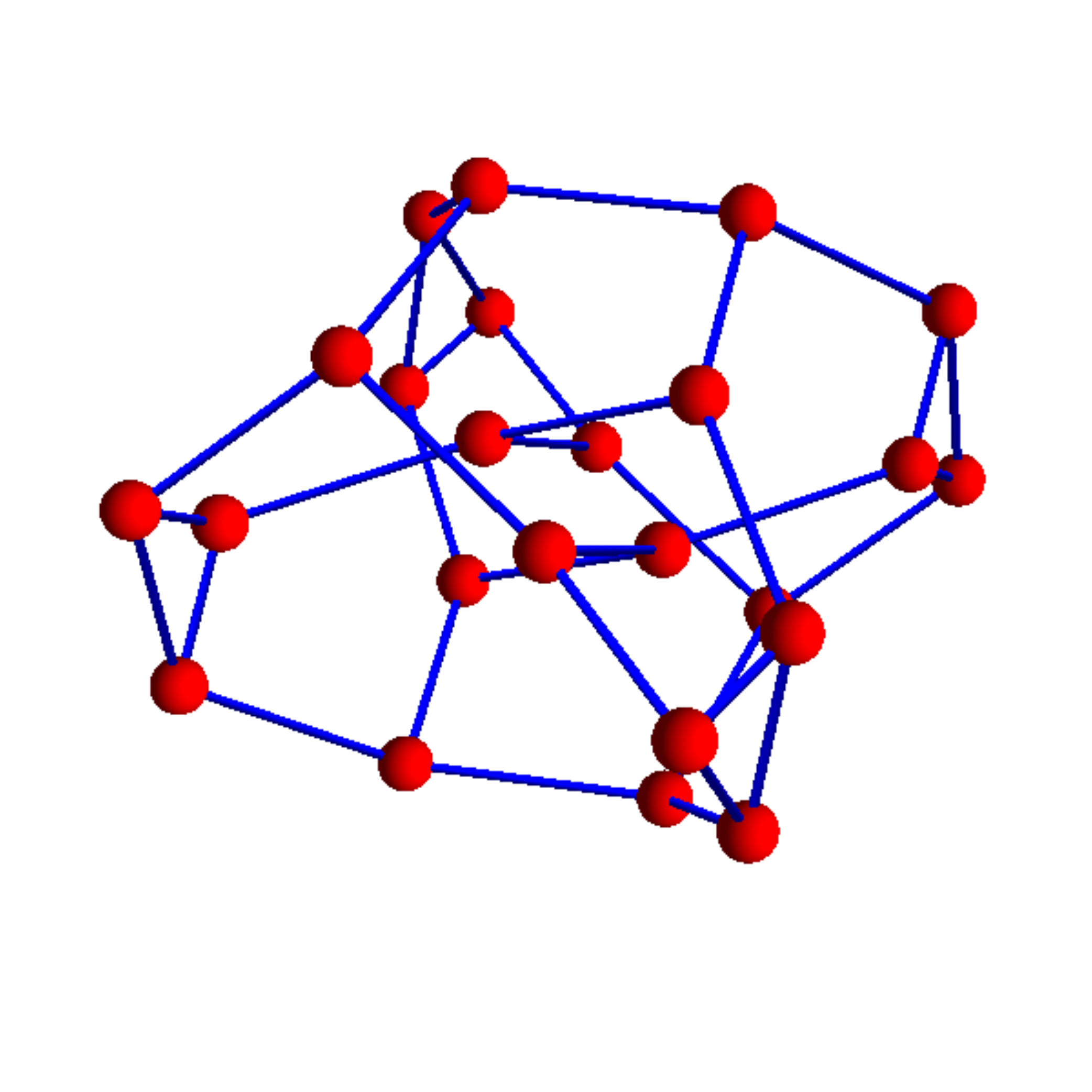}} }
\parbox{6.0cm}{ \scalebox{0.22}{\includegraphics{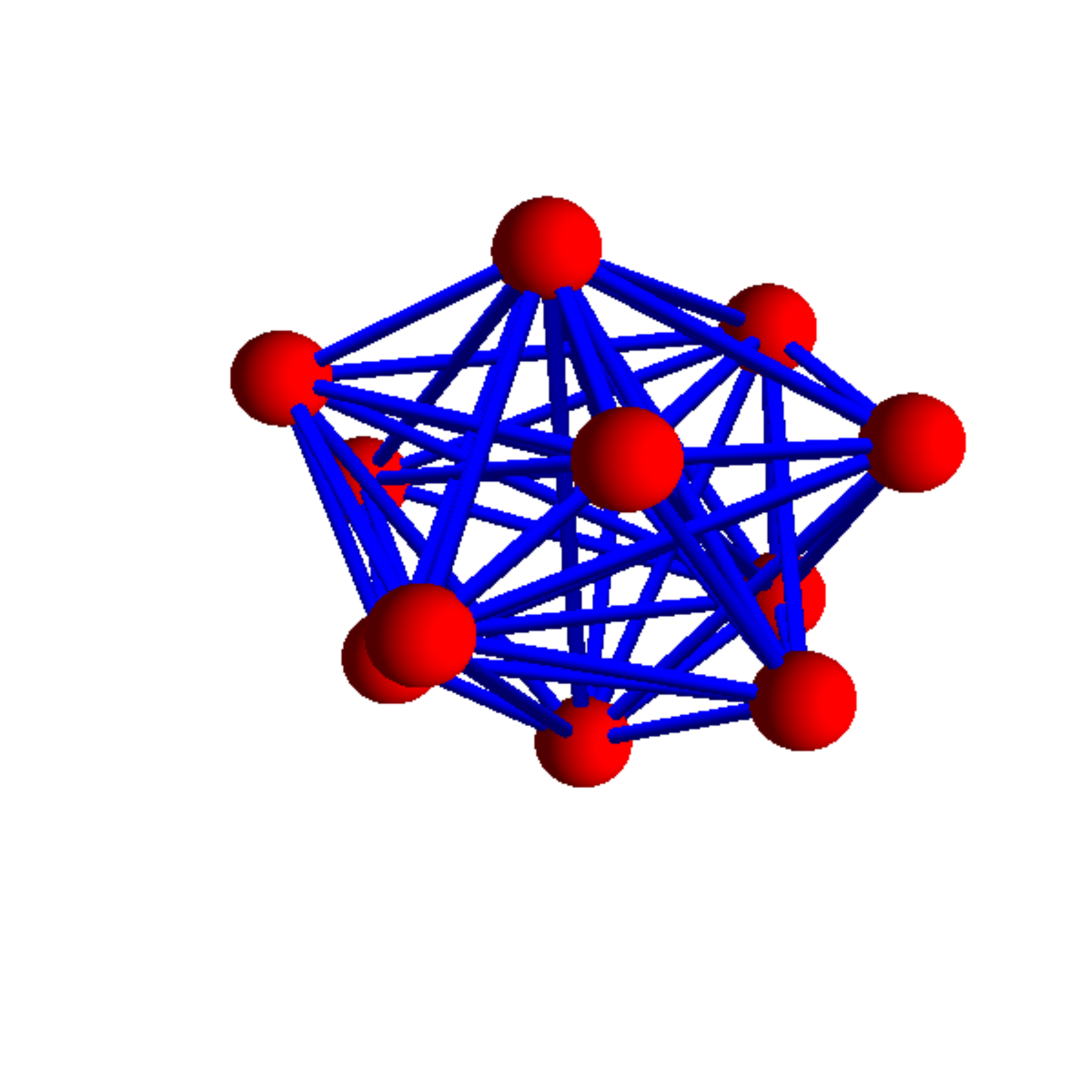}} }
\caption{
The Petersen graph $P(12,4)$ of dimension $4/3$ and Euler characteristic $-8$ can be seen
to the left. The curvatures take $12$ times the value $-1/2$ and $12$ times the value $-1/6$.
To the right, we see the Turan graph $T(13,4)$ of dimension $3$ and 
Euler characteristic $-23$. The curvatures take values $-2$ nine times and $-5/4$ four times. 
}
\label{petersenturan}
\end{figure}

We have studied the Gauss-Bonnet theme in a geometric setting 
for $d$-dimensional graphs for which unit spheres of a graph satisfy properties
familiar to unit spheres in $R^d$. In that case, the results look more similar to differential geometry
\cite{cherngaussbonnet}. \\

The curvature for three-dimensional graphs for example is zero everywhere and positive sectional 
curvature everywhere leads to definite bounds on the diameter of the graph. 
Also for higher dimensional graphs, similar than Bonnet-Schoenberg-Myers bounds assure in the continuum, 
positive curvature forces the graph to be of small diameter, 
allowing to compute the Euler characteristic in finitely many cases allowing in principle to 
answer Hopf type questions about the Euler characteristic of finite graphs with positive 
sectional curvature by checking finitely many cases. Such questions are often open in classical 
differential geometry but become a finite combinatorial problem in graph theory. 
Obviously and naturally, many question in differential geometry, whether variational, spectral or topological can be 
asked in pure graph theory without imposing any additional structure on the graph.  Dimension, 
curvature and Euler characteristic definitely play an important role in such quests. \\

\begin{figure}
\parbox{6.0cm}{ \scalebox{0.22}{\includegraphics{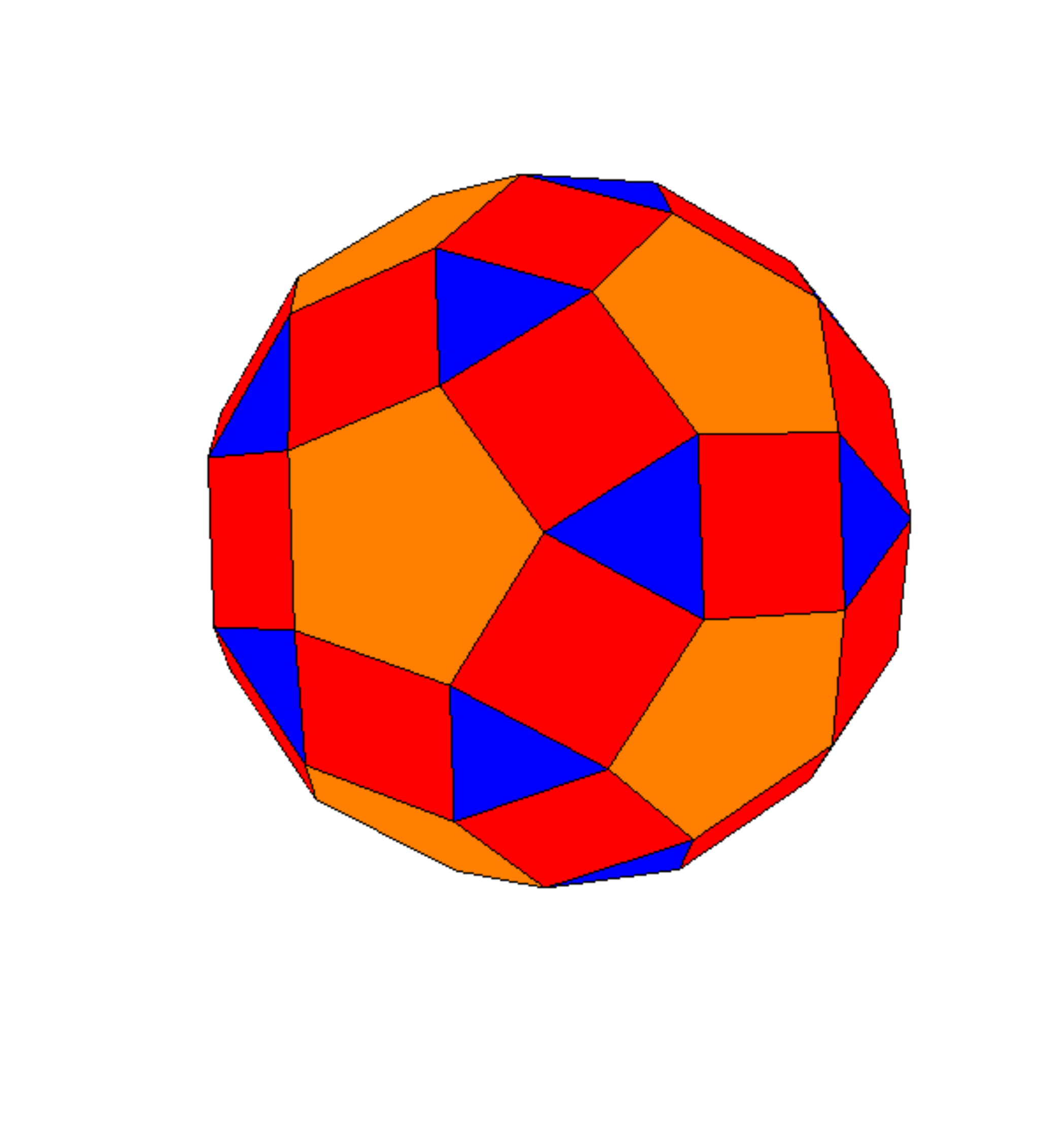}} }
\parbox{6.0cm}{ \scalebox{0.22}{\includegraphics{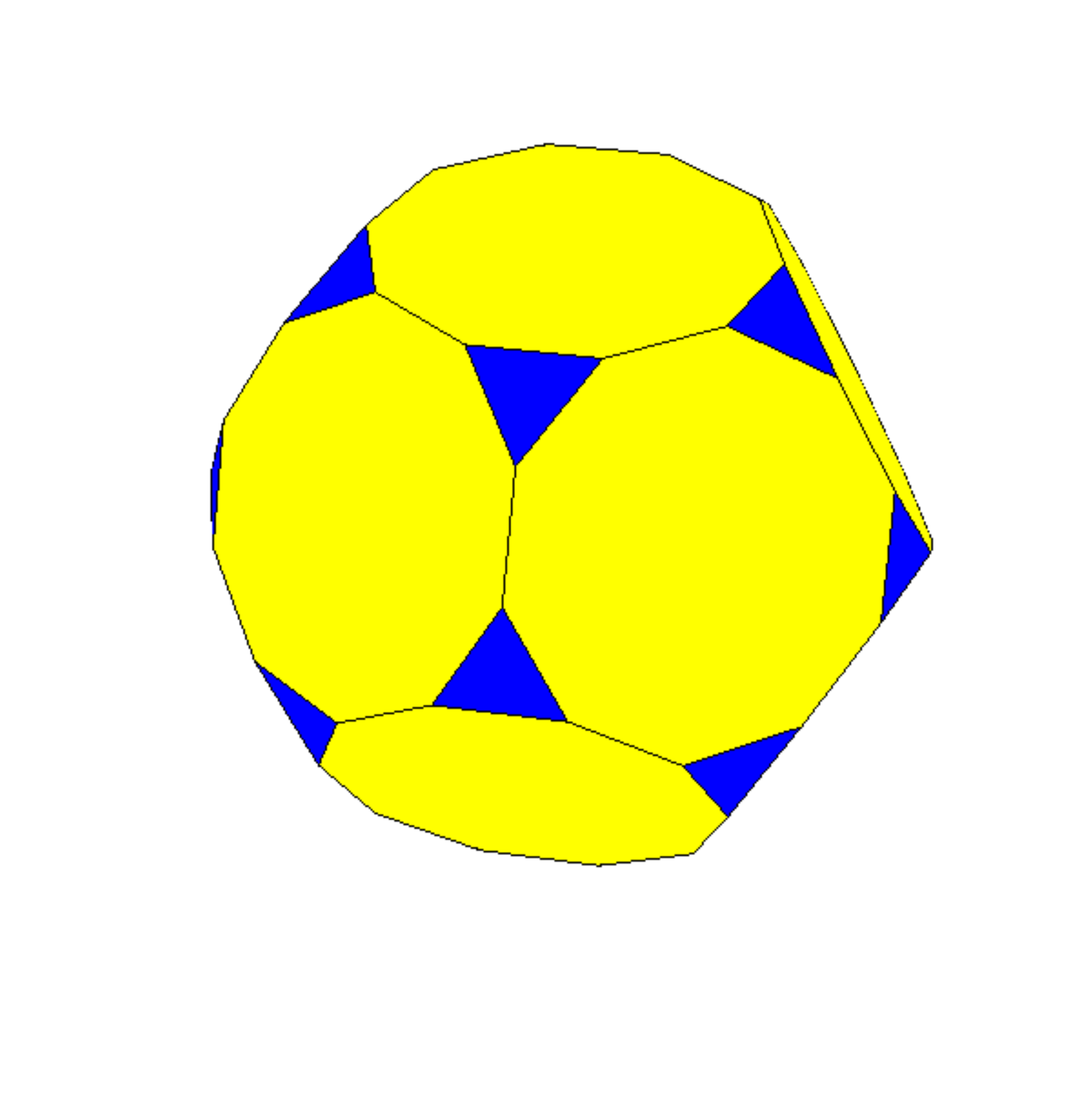}} }
\caption{
The dimension of the {\bf small rhombicosidodecahedron} is $3/2$: each point has dimension $3/2$ because
each unit sphere is a graph with $4$ vertices, where two points have dimension $1$ and two points have
dimension $0$. The unit sphere has dimension $1/2$. 
The dimension of the {\bf truncated dodecahedron} is $5/3$ because every 
unit sphere has dimension $2/3$.  }
\label{archimedean}
\end{figure}

What are the connections between dimension and Euler characteristic and curvature? We don't know much yet,
but there are indications of more connections: all flat connected graphs we have seen for example are geometric with 
uniform constant dimension like cyclic graphs or toral graphs. An interesting question is to study graphs where
Euler characteristic is extremal, where curvature (also Ricci or scalar analogues of curvature) 
is extremal or where curvature is constant. 
So far, the only connected graphs with constant curvature we know of are complete graphs $K_n$ with curvature 
$1/n$, cyclic graphs $C_n$ with curvature $0$, discrete graphs $P_n$ with 
curvature is $1$, the octahedron $O$ with curvature $1/6$, the icosahedron $I$ with curvature $1/12$,
higher dimensional cross polytopes with $2n$ vertices and curvature $1/(2n)$,
the 600 cell 120 vertices, where each unit sphere is an icosahedron and which is ``flat"
$$ K = \frac{V_{-1}}{1} - \frac{V_0}{2} + \frac{V_1}{3} - \frac{V_2}{4} 
     = \frac{1}{1} - \frac{12}{2} + \frac{30}{3} - \frac{20}{4} = 0 $$
like for any three dimensional geometric graph \cite{cherngaussbonnet},
twisted tori $K_{n,m}$ with curvature $0$ as well as higher dimensional regular tesselations of tori. 
We could not yet construct a graph with constant negative curvature
even so they most likely do exist. We start to believe that geometric graphs - for which local dimension is constant 
like the ones just mentioned - are the only connected constant curvature graphs. \\

We look in this article at connections with random graph theory 
\cite{bollobas,nbw2006}, an area of mathematics which has become useful for the 
study of complex networks \cite{CohenHavlin,newman2010,vansteen,ibe}. The emergence of interest in 
web graphs, social networks, neural networks, complex proteins or nano technology makes it an
active area of research.

\section{Random subgraphs of the complete graph}

We inquire in this section about the dimension of a typical graph in the probability space $G(n,1/2)$,
where each edge is turned on with probability $1/2$. To investigate this, we can look at all possible 
graphs on a fixed vertex set of cardinality $n$ and find the dimension expectation by computing 
the dimension of each graph and adding this up. When we counted dimensions for small $n$ by brute force,
we noticed to our surprise that the sum of all dimensions of subgraphs is an integer. 
Our limit for brute force summation was $n=7$, where we have already $2^{21}=2'097'152$ graphs. 
For the next entry $n=8$, we would have had to check $128$ times more graphs. \\

Note that as usual, we do not sum over all subgraphs of $K_n$ but all subgraphs of $K_n$ for which 
$|V|=v_0=n$. While this makes no difference for dimension because isolated points have dimension $0$,
it will matter for Euler characteristic later on, because isolated points have Euler characteristic $1$.
For $K_3$ for example, there are $3$ graphs with dimension $1$, one graph with  
dimension $2$ and $3$ graphs with dimension $2/3$. The sum of all dimensions is $7$. 
For $K_4$, the sum of dimensions over all $2^6=64$  subgraphs is $75$:
there are $22$ subgraphs of dimension $1$, there are $12$ of dimensions $5/3$ and $3/4$ each,
$6$ of dimensions $1/2$ and $2$ each, $4$ of dimension $3/2$  and one of dimension $0$ and $3$
each. Also for $K_5$, dimension $1$ appears most with $237$ subgraphs followed with $120$ graphs
of dimension $22/15$. For $K_6$ already, we have two integer champions: dimension $1$ appears for
$3'981$ subgraphs and dimension $2$ for $2'692$ subgraphs. The total sum of dimension is $53'275$.

\begin{theorem}[Average dimension on $G(n,1/2)$]
The average dimension on $G(n,1/2)$ satisfies the recursion
$$ d_{n+1} = 1+\sum_{k=0}^n \frac{\B{n}{k}}{2^n} d_k  \; , $$
where $d_0=-1$ is the seed for the empty graph. The sum over all dimensions 
of all order $n$ subgraphs of $K_n$ is an integer. 
\label{randomgraph}
\end{theorem}

\begin{proof}
Let $g(n)=2^{\B{n}{2}}$ be the number of graphs on the vertex set $V=\{ 1,...,n \; \}$,
and let $f(n)$ the sum of the dimensions of all subgraphs of the complete graph with $n$
vertices. We can find a recursion for $f(n)$ by adding a $(n+1)$'th vertex point $x$ and then 
count the sum of the dimensions over all subgraphs of $K_{n+1}$ 
vertices by partitioning this set of subgraphs up into the set $Y_k$ which have $k$ edges connecting
the old graph to the new vertex. There are $\B{n}{k}$ possibilities to build such 
connections. In each of these cases, the unit sphere $S(x)$ is a complete
graph of $k$ vertices. We get so the sum $\sum_{k=0}^n (f_k+g(k)) \B{n}{k}$ of 
dimensions of such graphs because we add $1$ to each of the $\B{n}{k}$ cases. 
From this formula, we can see that the sum of dimensions is an integer. 
The dimension itself satisfies the recursion
$$ d_{n+1} = 2^{\B{n}{2}-\B{n+1}{2}} \sum_{k=0}^n (d_k + 1) \B{n}{k} \; . $$ 
With $\B{n}{2}-\B{n+1}{2} = -n$ and using $2^{-n} \sum_{k=0}^n \B{n}{k} = 1$, we get
the formula.
\end{proof}

{\bf Remarks.} \\
{\bf 1.} We see that $d_{n+1}$ is $1$ plus a Bernoulli average of the sequence $\{ d_k \; \}_{k=0}^n$.  \\
{\bf 2.} Theorem~\ref{randomgraph} will be generalized in Theorem~\ref{randomgraph2} from $p=1/2$ to general $p$. 

\section{Average dimension on classes of graphs}

We can look at the average dimension on subclasses of graphs. For example, what is
the average dimension on the set of all $2$-dimensional graph with $n$ vertices? This is not so easy
to determine because we can not enumerate easily all two-dimensional graphs of order $n$. 
Only up to discrete homotopy operations, 
the situation for two dimensional finite graphs is the same as for two dimensional 
manifolds in that Euler characteristic and the orientation determines the equivalence class.  \\

We first looked therefore at the one-dimensional case. Also here,
summing up the dimensions of all subgraphs by brute force showed that the sum is always 
an integer and that the dimension is constant equal to $3/4$. This is true for a general
one dimensional graph without boundary, graphs which are
the disjoint union of circular graphs. 

\begin{theorem}[3/4 theorem]
For a one-dimensional graph $G=(V,E)$ without boundary, the average dimension of all 
subgraphs of $G$ is always equal to $3/4$. The sum of the dimensions of subgraphs $H$
of $G$ with the same vertex set $V$ is an integer. 
\end{theorem}
\begin{proof}
A one-dimensional graph without boundary is a finite union of cyclic graphs. 
For two disjoint graphs $G_1,G_2$ with union $G=G_1 \cup G_2$, we have 
${\rm dim}(G) = ({\rm dim}(G_1) |G_1| + {\rm dim}(G_2) |G_2|)/|G|$. It is therefore enough to 
prove the statement for a connected circular graph of size $n \geq 4$. This can be done 
by induction. We can add a vertex in the middle of one of the edges to get from 
$C_n$ to $C_{n+1}$. The dimensions of the other vertices do not change. 
The new point has dimension $0$ with probability $1/4$ and $1$ with probability $3/4$.  
Since the smallest one dimensional graph has $4$ nodes, and $2^{\B{n}{2}}$ is already a multiple
of $4$, the sum dimensions of all subgraphs is an integers.
\end{proof}

{\bf Remarks.} \\
{\bf 1.} We will generalize this result below and show that for a one dimensional circular graph,
the expected dimension of a subgraph is $p(2-p)$. The $3/4$ result is the special case when $p=1/2$. 
The result is of course different for the triangle $K_3=C_3$, which is two dimensional and for which the 
expected dimension is $p(2-p+p^2)$. We will call the function $f(p) = p(2-p)$ the signature function. 
It is in this particular case the same for all one dimensional graphs without boundary. \\
{\bf 2.}  Is the sum of dimensions of subgraphs of a graph with integer dimension and constant degree
an integer? No. Already for an octahedron, a graph $G$ of dimension $2$ which has $2^{12}=4096$ different subgraphs, 
a brute force computation shows that the sum of dimensions is $a=15424/3$ and 
that the average dimension of a subgraph of $G$ is $a/2^{12}=1.25521$. 
The unit ball $B_1(v)$ in the octahedron is the wheel graph $W_4$ with four spikes
in which the sum of the dimensions is $284$ and the average dimension is $284/256=1.10938$. \\
{\bf 3.} We would like to know the sum of all dimensions of subgraphs for flat tori, finite graphs for which each 
unit disc is a wheel graph $W_6$. Such tori are determined by the lengths $M,K$ of the smallest homotopically
nontrivial one dimensional cycles as well as a "Dehn" twist parameter for the identification. 
Already in the smallest case $M=K=4$, the graph has $48$ edges and summing 
up over all $2^{\B{48}{2}}=3.6 \cdot 10^{339}$ possible subgraphs is impossible. It becomes 
a percolation problem \cite{Grimmet}. It would be interesting to know for example
whether the dimension signature functions $f_{K,M}(p) = {\rm E}_p[{\rm dim}]$ 
have a continuous limit for ${\rm min}(K,M) \to \infty$. \\ 
{\bf 4.} For the wheel graph $W_6$, the unit disc in the flat torus, 
the average dimension is $(159368/35)/2^{12} = 1.1116..$. 
For a flat torus like $T_{5,5}$ we can not get the average dimension exactly but measure it to be about $1.3$. 
We observe in general however that the dimension depends smoothly on $p$. \\
{\bf 5.} Brute force computations are not in vain because they allow also to look at the distribution 
of the dimensions of graphs with $n$ vertices. Since we do not yet have 
recursive formulas for the higher moments, it is not clear how this behaves in the limit
$n \to \infty$. 

\begin{figure}
\parbox{15cm}{
\parbox{7.3cm}{ \scalebox{0.22}{\includegraphics{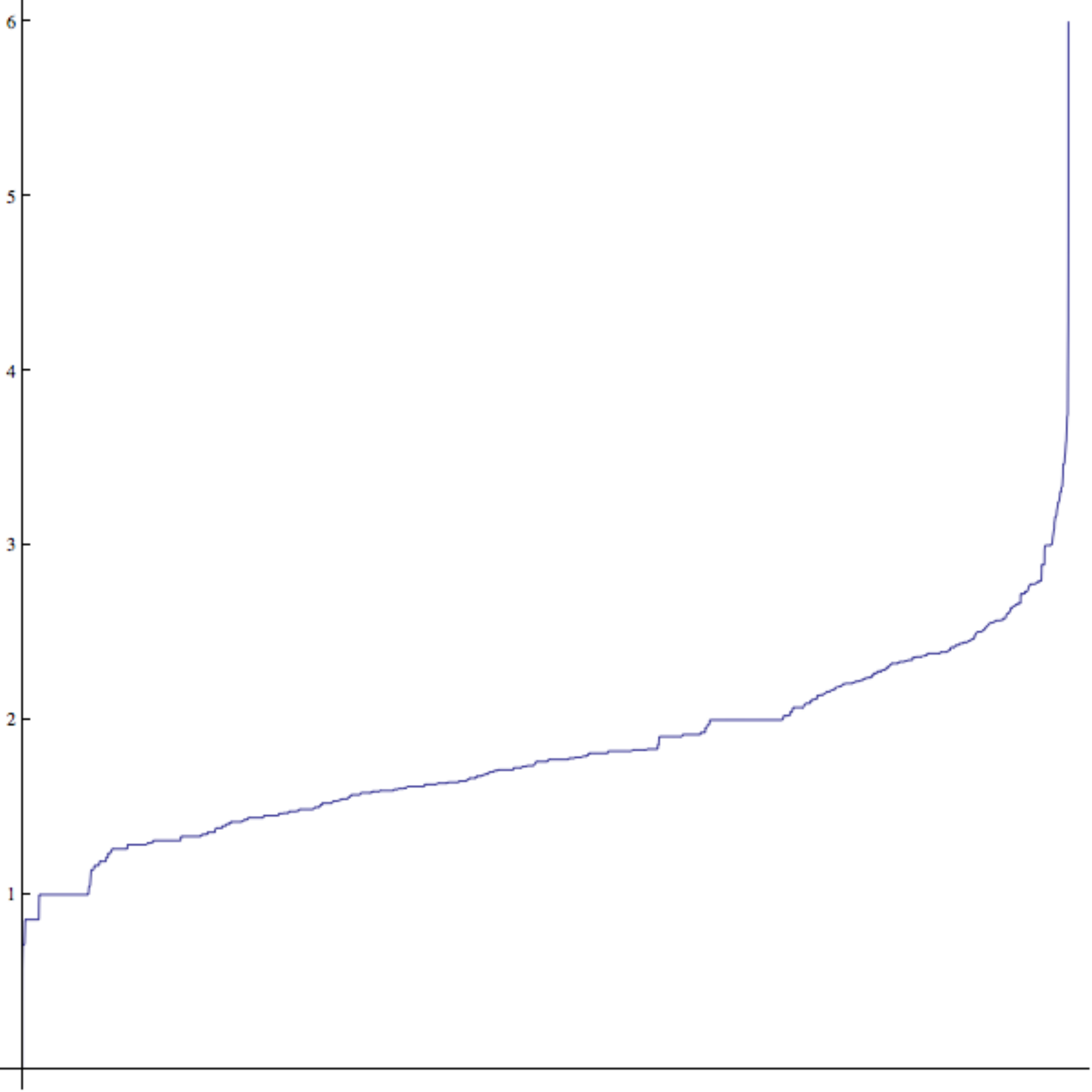}} }
\parbox{7.3cm}{ \scalebox{0.22}{\includegraphics{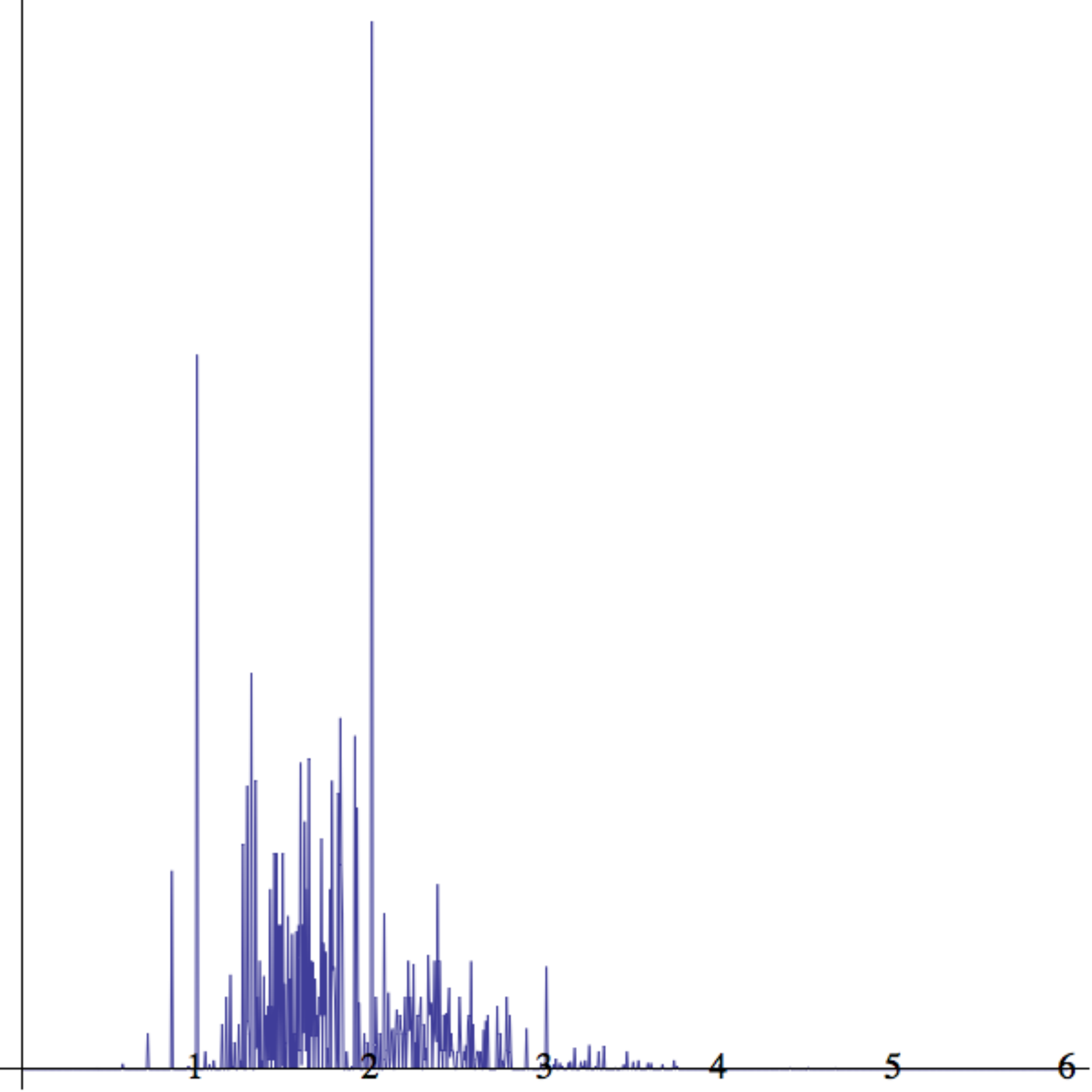}} }
}
\caption{
The distribution of the random variable ${\rm dim}$ on the finite probability 
space $G(7,1/2)$ of all graphs with vertex cardinality $7$. The left picture shows the cumulative 
distribution function, the right the density function. Graphs with integer dimension 
appear to have higher probability than graphs with fractional dimension. This is a phenomenon,
we also see in concrete networks, like social graphs or computer networks. 
The function ${\rm dim}$ takes 245 different dimension values on $G(7,1/2)$. The dimension 2 appears for
$4'146'583$  graphs, the next frequent dimension is $1$ appears in $99'900$ graphs.
Only in third rank is dimension $55/42$ is a fraction. It appears for $55'440$ graphs. } 
\label{graph7}
\end{figure}

\begin{figure}
\parbox{15cm}{
\parbox{6.5cm}{ \scalebox{0.25}{\includegraphics{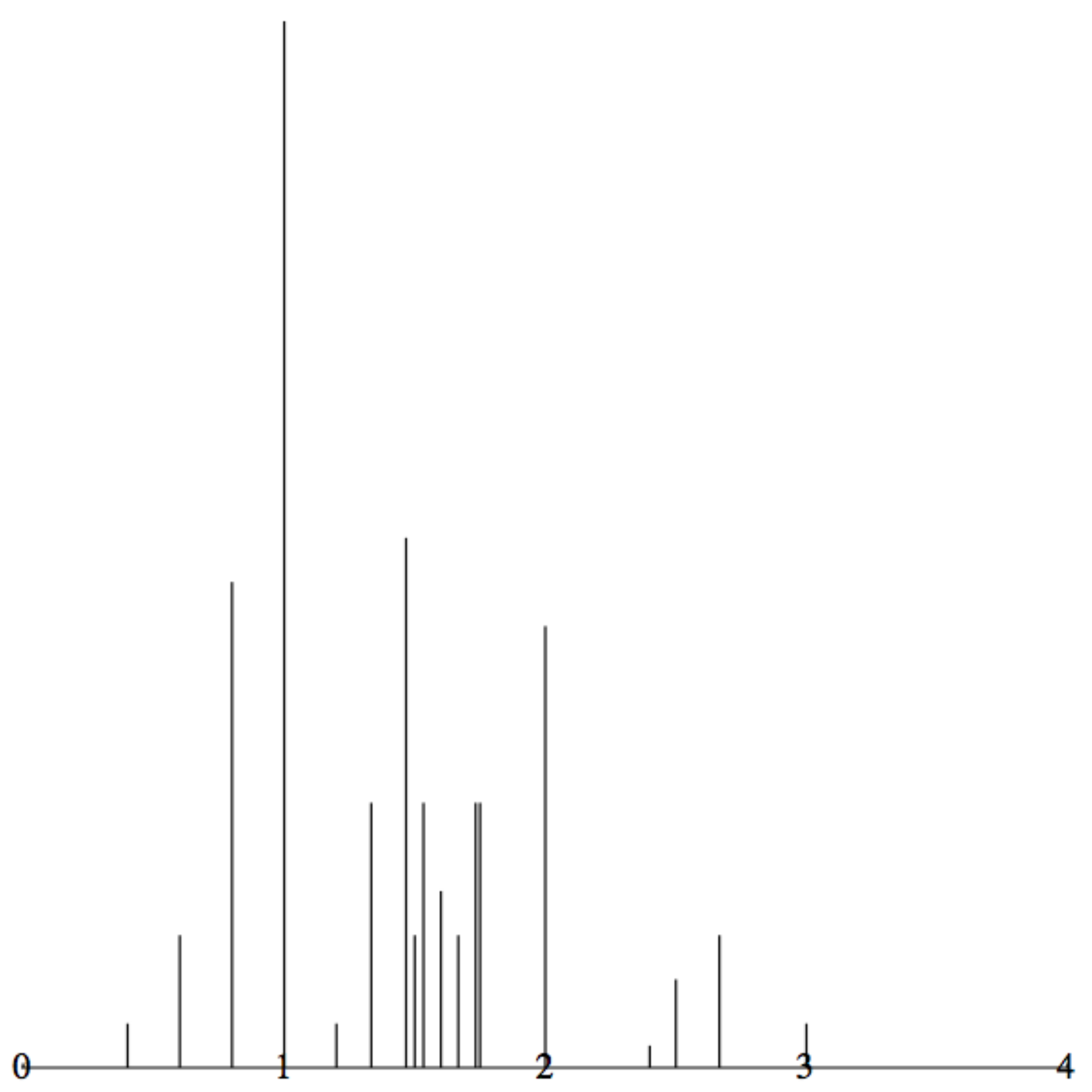}} }
\parbox{6.5cm}{ \scalebox{0.25}{\includegraphics{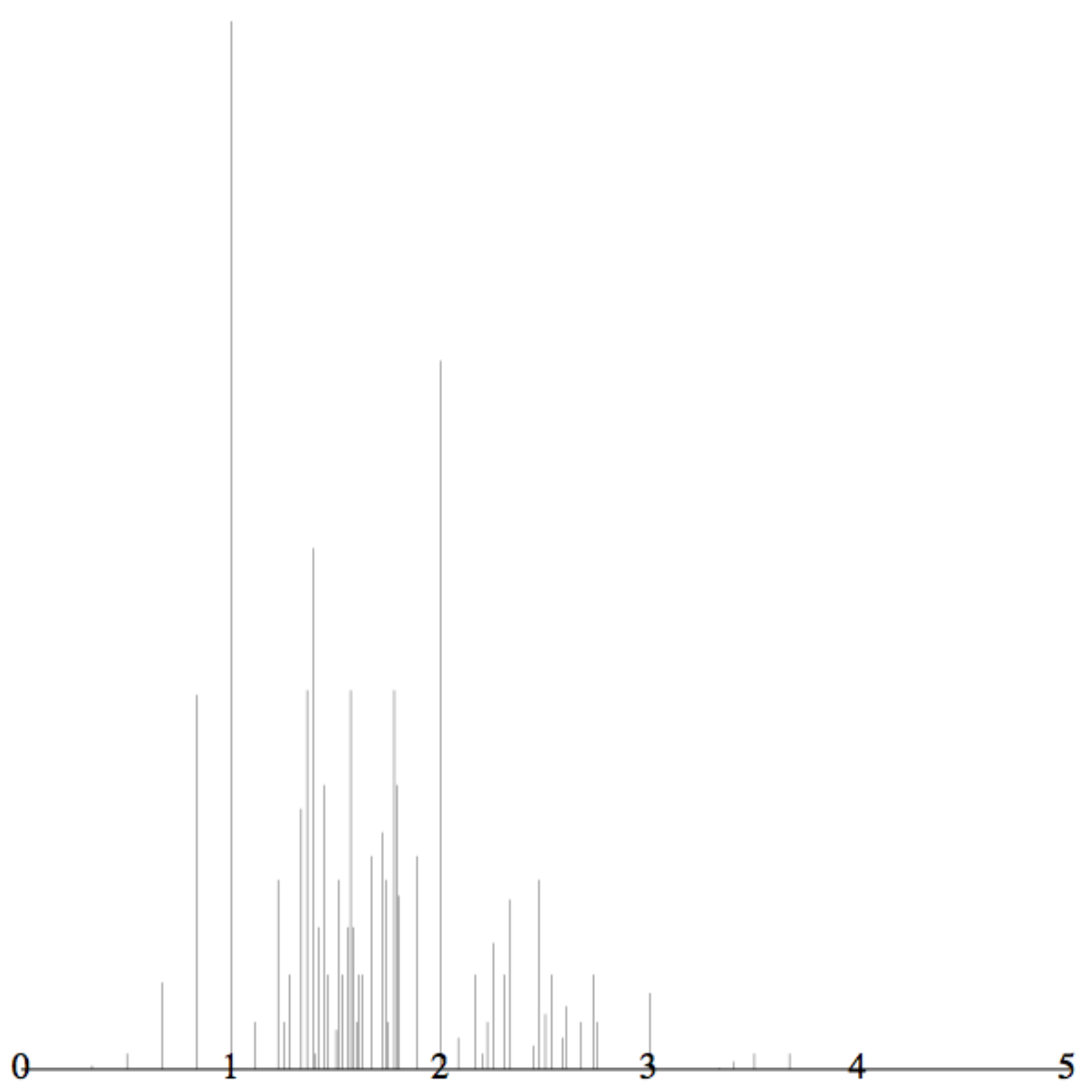}} }
}
\caption{
The dimension statistics on $G(5,1/2)$ and $G(6,1/2)$ already indicates a preference
for integer dimension of subgraphs. }
\end{figure}

The first sums of dimensions are
$s_1=0$,
$s_2=1$,
$s_3=7$,
$s_4=75$, 
$s_5=1451$,
$s_6=53275$,
$s_7=3791451$,
$s_8=528662939$,
$s_9=145314459035$,
$s_{10}=79040486514843$, and $s_{11} = 85289166797880475$.  On a graph with $5$ vertices for example, 
the sum of all dimensions over all the $2^{10}$ subgraphs is $1451$. 
We could not possibly compute $s_{11}$ by brute force, since there are just too many graphs.  \\

Favoring integer dimensions for concrete or random graphs is a resonance phenomena 
of number theoretical nature. Whether it is a case for "Guy's law of small numbers" disappearing in
the limit $n \to \infty$ remains to be seen. 

\section{The dimension of a random p-percolating graph} 

We generalize now the probability measure on the space of all graphs with $n$ elements 
and switch each edge on with probability $0 \leq p \leq 1$. This is the classical Erd\"os-R\'enyi model
\cite{erdoesrenyi59}. With the probability measure $P_p$ on $X_n$, the probability space is called
$G(n,p)$. For this percolation problem on the complete graph, the mean degree
is ${\rm E}[{\rm deg}] = n p$.  \\

The following result is a generalization of 
Theorem~\ref{randomgraph}, in which we had $p=1/2$. It computes $d_n = {\rm E_{n,p}}[{\rm dim}]$.

\begin{theorem}[Average dimension on $G(n,p)$]
The expected dimension ${\rm E}_p[{\rm dim}]$ on $G(n,p)$ satisfies
$$ d_{n+1}(p) = 1+\sum_{k=0}^n \B{n}{k} p^k (1-p)^{n-k} d_k(p) \; , $$
where $d_0=-1$. Each $d_n$ is a polynomial in $p$ of degree $\B{n}{2}$. 
\label{randomgraph2}
\end{theorem}

\begin{proof}
The inductive derivation for $p=1/2$ generalizes:
add a $n+1$'th point $P$ and partition the number of graphs into sets $Y_k$, 
where $P$ connects to a $k$-dimensional graph within the old graph. The expected
dimension of the new point is then 
$$ d(n+1) = \sum_{k=1}^n \B{n}{k} p^k (1-p)^{n-k} (d(k)+1)  \;  $$
and this is also the expected dimension of the entire graph. This can be written as 
$$ d(n+1) = \sum_{k=1}^n \B{n}{k} p^k (1-p)^{n-k} d(k) + \sum_{k=1}^n \B{n}{k} p^k (1-p)^{n-k} 1  \;  $$
which is 
$$ d(n+1) = \sum_{k=0}^n \B{n}{k} p^k (1-p)^{n-k} d(k) +1 - (1-p)^n d(0) - (1-p)^n $$
which is equivalent to the statement. 
\end{proof}

Again, if we think of the vector $d=(d_0,d_1, \dots ,d_n)$ 
as a random variable on the finite set $\{0,1, \dots ,n \; \}$ then $d_{n+1}$ is 
$1$ plus the expectation of this random variable with respect to the Bernoulli distribution
on $\{0,1, \dots ,n \; \}$ with parameters $n$ and $p$.  \\

Lets look at the first few steps:
we start with $-1$, where the expectation is $-1$ and add $1$ to get $0$. Now we have 
$(-1,0)$ and compute the expectation of this 
$(-1) p^0 (1-p)^1 + 0 p^1 (1-p)^1=p-1$. Adding $1$ gives $p$
so that we have the probabilities $(-1,0,p)$. Now compute the expectation again with
$(-1) p^0 (1-p)^2 + 2 p^1 (1-p)^1 \cdot 0 + p^2 (1-p)^0 p$. Adding $1$ gives 
the expected dimension of $d_3(p) = 2p - p^2 + p^3$ on a graph with 3 vertices. 
We have now the vector $(-1,0,p,2p-p^2+p^3)$. To compute the dimension expectation 
on a graph with $4$ vertices, we compute the expectation 
$(-1) p^0 (1-p)^3 + 3 p^1 (1-p)^2 \cdot 0 + 3 p^2 (1-p)^1 \cdot p + 1 p^3 (1-p)^0 \cdot (2p-p^2+p^3) 
  = -1 + 3 \cdot p - 3 \cdot p^2 + 4 \cdot p^3 - p^4 - p^5 + p^6$
and add 1 to get the expected dimension 
$d_4(p) = 3 \cdot p - 3 \cdot p^2 + 4 \cdot p^3 - p^4 - p^5 + p^6$ on a graph of $4$ elements. 
Here are the first polynomials:

\begin{center}
\begin{tabular}{l}
$d_2(p)=p$\\ 
$d_3(p)=2p-p^2+p^3$\\
$d_4(p)=3p-3p^2+4p^3-p^4-p^5+p^6$\\
$d_5(p)=4p-6p^2+10p^3-5p^4-3p^5+5p^6-p^8-p^9+p^{10}$\\
\end{tabular}
\end{center}

\begin{figure}
\parbox{6.0cm}{ \scalebox{0.22}{\includegraphics{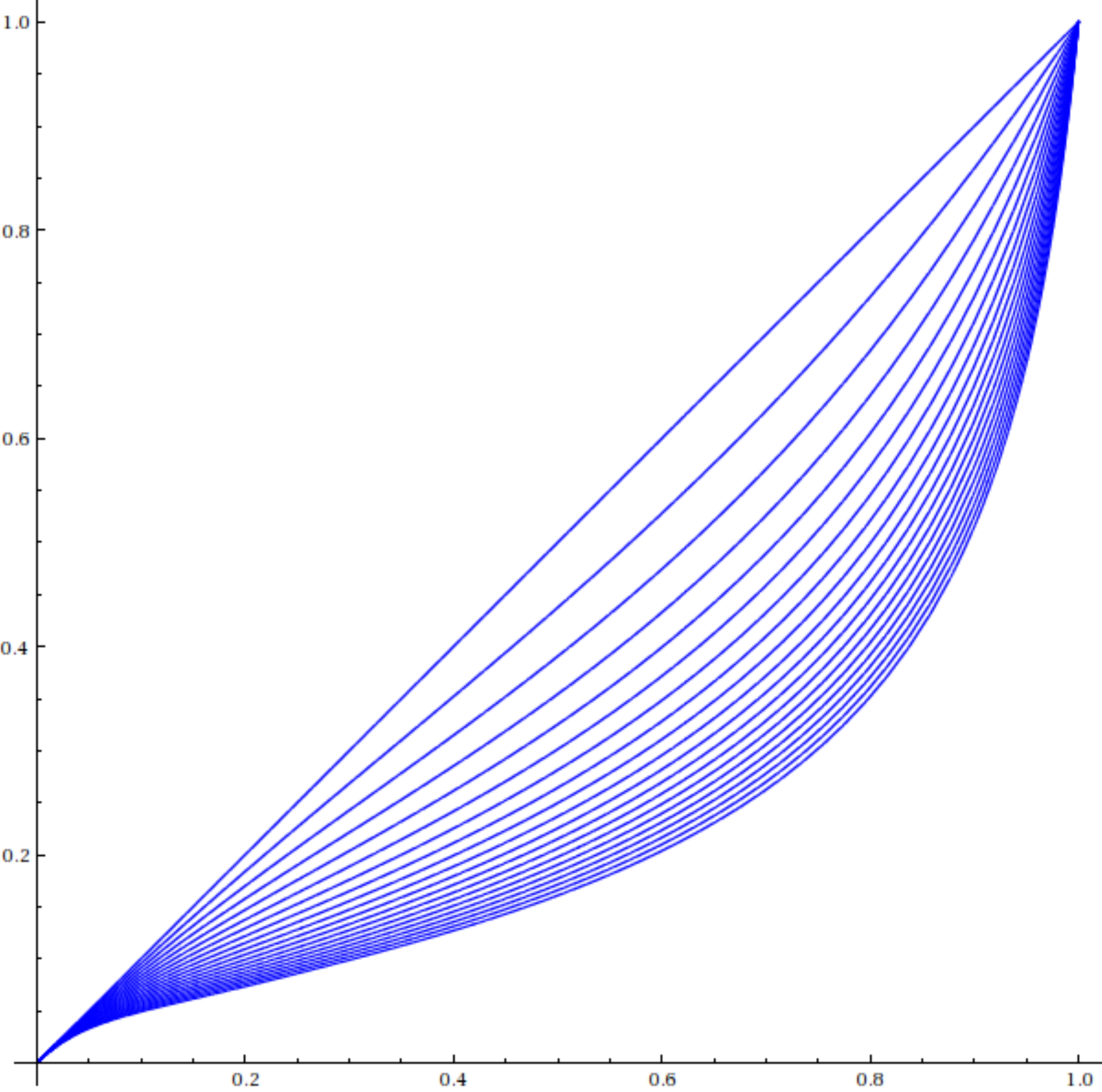}} }
\caption{
The expected dimension of a graph as a function of $p$ is monotone and is shown here for $n=1, \dots , 20$. 
We see the graphs of the functions $d_n(p)/(n-1)$ for $n=1,2, \dots 20$. A division by $(n-1)$ produces
a scaling so that $d_n(1) = 1$.
We experimentally see that they satisfy a power law: there are functions constants $a=a(p)$ and
$c=c(p)$ for which with the scaling law $d_n(p) \sim c/n^a$ holds in the limit $n \to \infty$.
Especially $d_n(p)/(n-1)$ is monotone in $n$ for every $p$. This needs to be explored more.}
\label{pddependence}
\end{figure}

\begin{figure}
\parbox{4.0cm}{ \scalebox{0.22}{\includegraphics{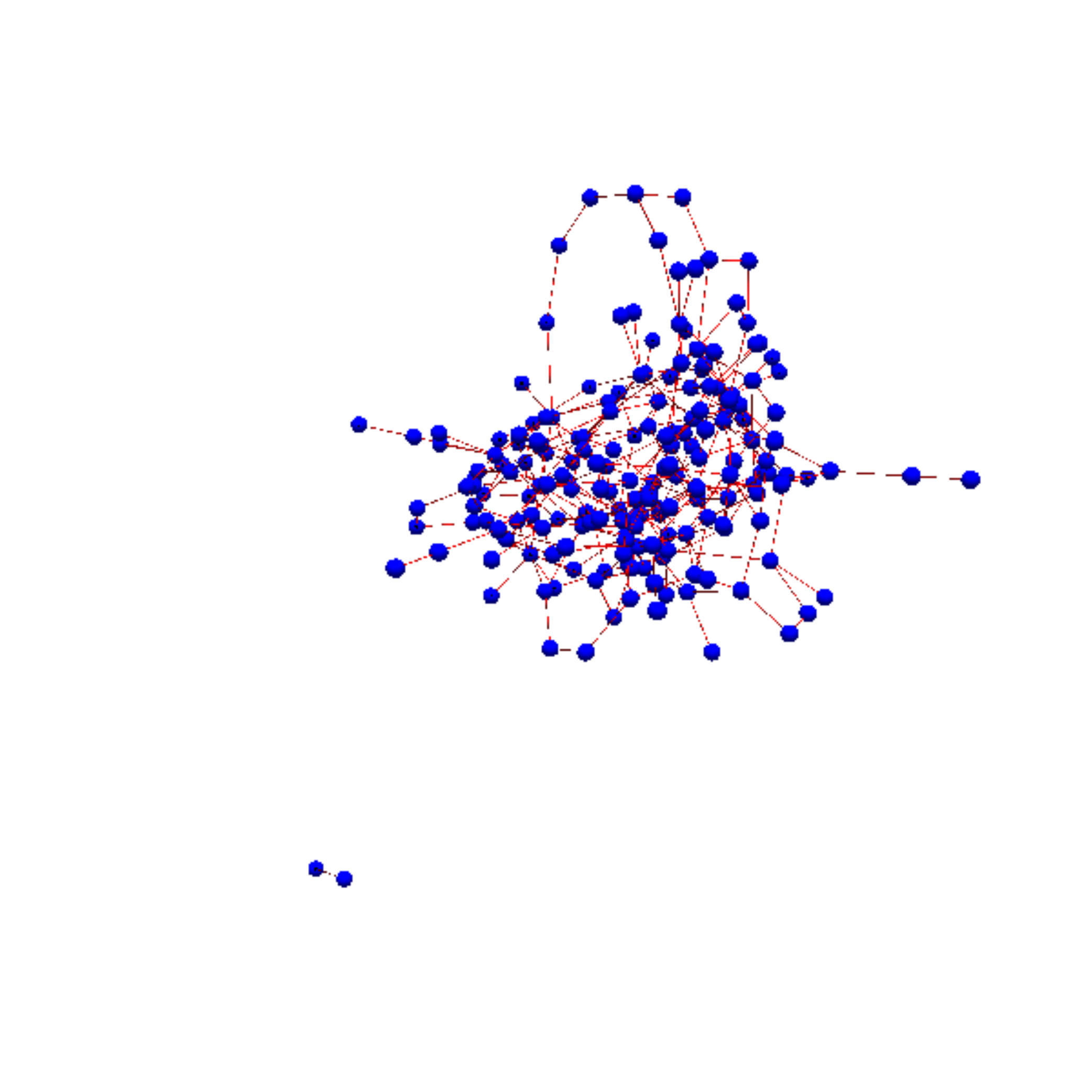}} }
\parbox{4.0cm}{ \scalebox{0.22}{\includegraphics{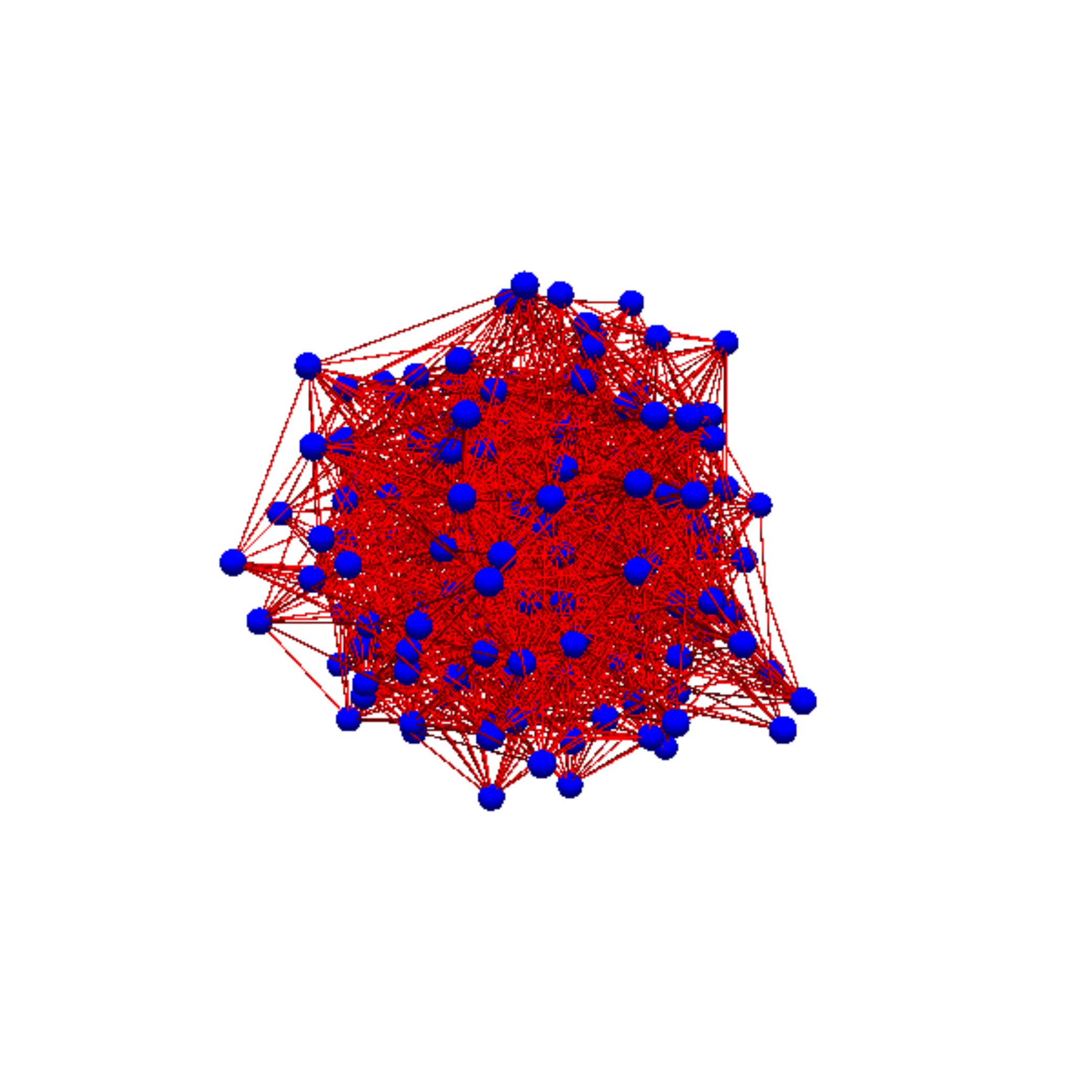}} }
\parbox{4.0cm}{ \scalebox{0.22}{\includegraphics{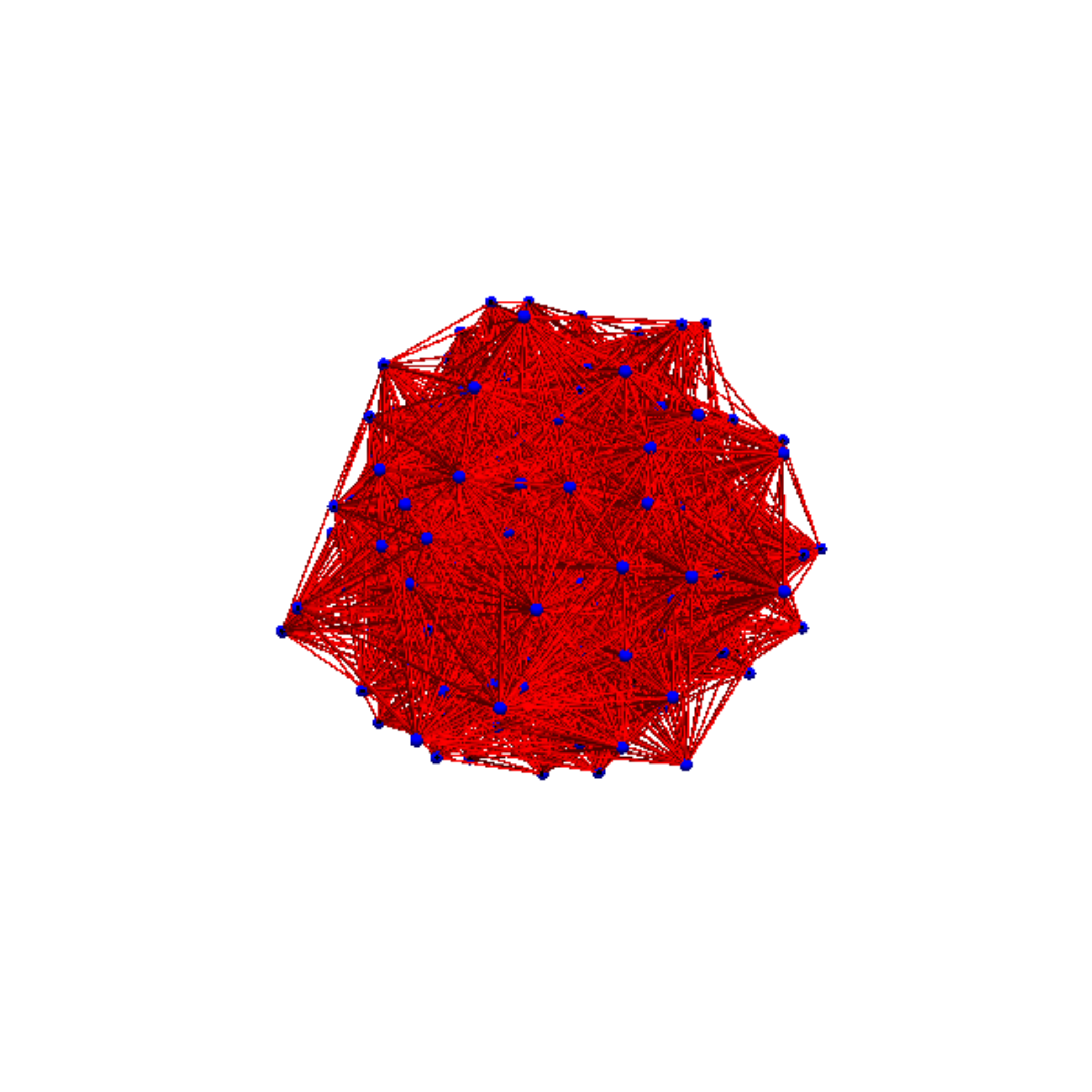}} }
\caption{
Three random graphs in which the percolation probability $p$ parameter is
chosen to have an expected dimension $1$, or $2$ or $3$. 
}
\label{pddependence2}
\end{figure}

The expected number of cliques $K_k$ of the complete graph $K_n$ is $\B{n}{k} p^{\B{k}{2}}$.
By approximating the Binomial coefficients $\B{n}{k}$ with the Stirling formulas, 
Erd\"os-R\'enyi have shown (see Corollary 4 in \cite{erdoesrenyi61}), that
for $p < n^{-2/(k-1)}$, there are no $K_k$ subgraphs in the limit $n \to \infty$
which of course implies that then the dimension is $\leq (k-1)$ almost surely.
For $p<n^{-2/2}$, there are no triangles in the limit and ${\rm dim}(G) \leq 1$
for $p<n^{-2/3}$, there are no tetrahedra in the limit and ${\rm dim}(G) \leq 2$.  \\
 
Is there a threshold, so that for $p<n^{\alpha}$ the expectation of dimension converges? 
We see experimentally for any $p$ that ${\rm E}[{\rm dim}] \sim c(p)/n^{a(p)}$, where $c,a$ 
depend on $p$.  \\

\begin{figure}
\parbox{6.0cm}{ \scalebox{0.22}{\includegraphics{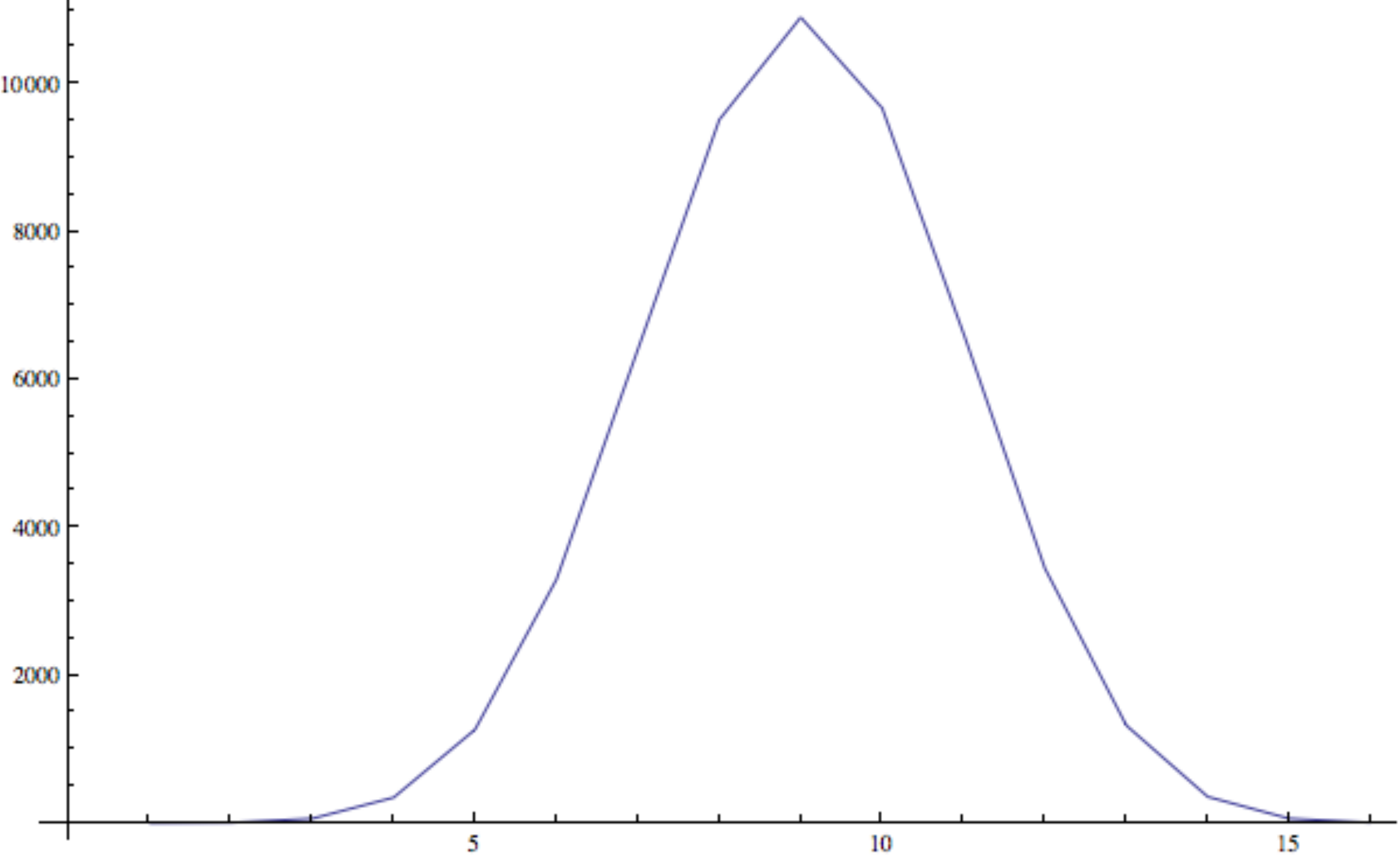}} }
\caption{
The sum $s_n(k)$ of dimensions of size-$k$ and order-$n$ subgraphs of a 
complete graph $K_n$ is an integer too.  
The sequence $s_n(k)$ appears to converge to a Gaussian distribution. 
This is not surprising, when considering the de Moivre-Laplace 
limit theorem and that dimension is highly correlated with the 
size of the graph. We see in experiments, that this distribution 
limit is achieved also for subgraphs of other graphs like circular graphs. 
}
\label{degreedimension}
\end{figure}

Finally lets look at a generalization of the $3/4$ theorem:

\begin{theorem}[Expected dimension of a subgraph of a one dimensional graph]
The expected dimension ${\rm E}_p[{\rm dim}]$ on all subgraphs of any one-dimensional graph
without boundary is $p(2-p)$. 
\end{theorem}
\begin{proof}
Proceed by induction. Add an other vertex in the middle of a given edge. 
The expectation of the dimension of the remaining points does not change. 
The expectation of the dimension of the new point is 
$1 \cdot p^2 + 2 p (1-p) + 0 (1-p)^2 = p(2-p)$ because the dimension of 
the point is one if both one of the two adjacent edges are present and $0$ 
if none is present. 
\end{proof}

{\bf Remark.}
More interesting is to introduce the variable $q=p=1$ and write  the expectation as a sum 
$$  \sum_{k=0}^n a_{k,n} \B{n}{k} q^{n-k} p^k \; , $$
where $a_{k,n}$ is the expectation of the dimension on the probability space $G(p)$ of 
all graphs $(V,E_k)$, where $E_k$ runs over all subsets of the vertices of the host graph $G$. 
For example, for the host graph $C_8$ we have
$$ {\rm E}[{\rm dim}] = 2 q^7 p  + 13 q^6 p^2 + 36 q^5 p^3 + 55 q^3 p^4 
  + 50 q^3 p^5 + 27 q^2 p^6 + 8 q^1 p^7 + p^8 = p(2-p)  \;  $$
from which we can deduce for example that $27/28$ is the expected dimension if two links
are missing in a circular graph $C_8$ or $13/28$ if two links are present and 
that $2/8$ is the expected dimension if only one link is present. The polynomial has the form
$$ 2 p q^{n-1} + (2n-3) q^{n-2} p^2 + n^2 q^{n-3} p^3 + \cdots + n q p^{n-1} + p^n \; . $$

\section{The Euler Characteristic of a random graph}

Besides the degree average ${\rm deg}$ and dimension ${\rm dim}$, 
an other natural random variable on the probability space $G(n,p)$ 
is the Euler characteristic $\chi$ defined in Equation~(\ref{eulercharacteristicdef}). \\

For $p=1/2$, we sum up the Euler characteristics over all  order $n$
sub graphs of the complete graph of order $n$ and then average. The list of sums starts with 
$$   1, 3,  13,    95,      1201, 25279,... $$ 
leading to expected Euler characteristic values  
$$   1, 1.5, 1.625, 1.48438, 1.17285, 0.771454, 0.34842, -0.0399132 \; . $$
Already for $n=5$, there are some subgraphs with negative Euler characteristics. 
For $n=8$, the expectation of $\chi$ has become negative for the first time. It will do so again and again.
These first numbers show by no means any trend: the expectation value of $\chi$ will oscillate 
indefinitely between negative and positive regimes and grow in amplitude. This follows from
the following explicit formula:

\begin{theorem}[Expectation of Euler characteristic]
The expectation value of the Euler characteristic on $G(n,p)$ is 
$$ {\rm E}_{n,p}[\chi] = \sum_{k=1}^n (-1)^{k+1} \B{n}{k} p^{\B{k}{2}} \; . $$
\end{theorem}
\begin{proof}
We only need expectation of the random variable $v_k$ on $G(n,p)$. But this is well known:
(see \cite{bollobas}). We have
$$  {\rm E}[v_k] = \B{n}{k+1} p^{\B{k+1}{2}}  .  $$  
This later formula is proven as follows: if $S_k$ is the set of 
all k-cliques $K_k$ in the graph. Now count the number $n(s)$ of times that a 
k-clique $s$ appears. Then ${\rm E}[v_k] 2^{\B{n}{2}} = \sum_{s \in S_k} n(s) = \B{n}{k} p^k$
because $n(s)$ is constant on $S_k$. 
We especially have ${\rm E}[v_0]=n,{\rm E}[v_1]=\B{n}{2} p$ is the expected value of the number of edges.
and ${\rm E}[v_2] = \B{n}{3} p^3$ is the expected number of triangles in the graph and
${\rm E}[v_3] = \B{n}{4} p^6$ is the expected number of tetrahedra $K_4$ in the graph. 
\end{proof} 

\begin{figure}
\parbox{6.0cm}{ \scalebox{0.35}{\includegraphics{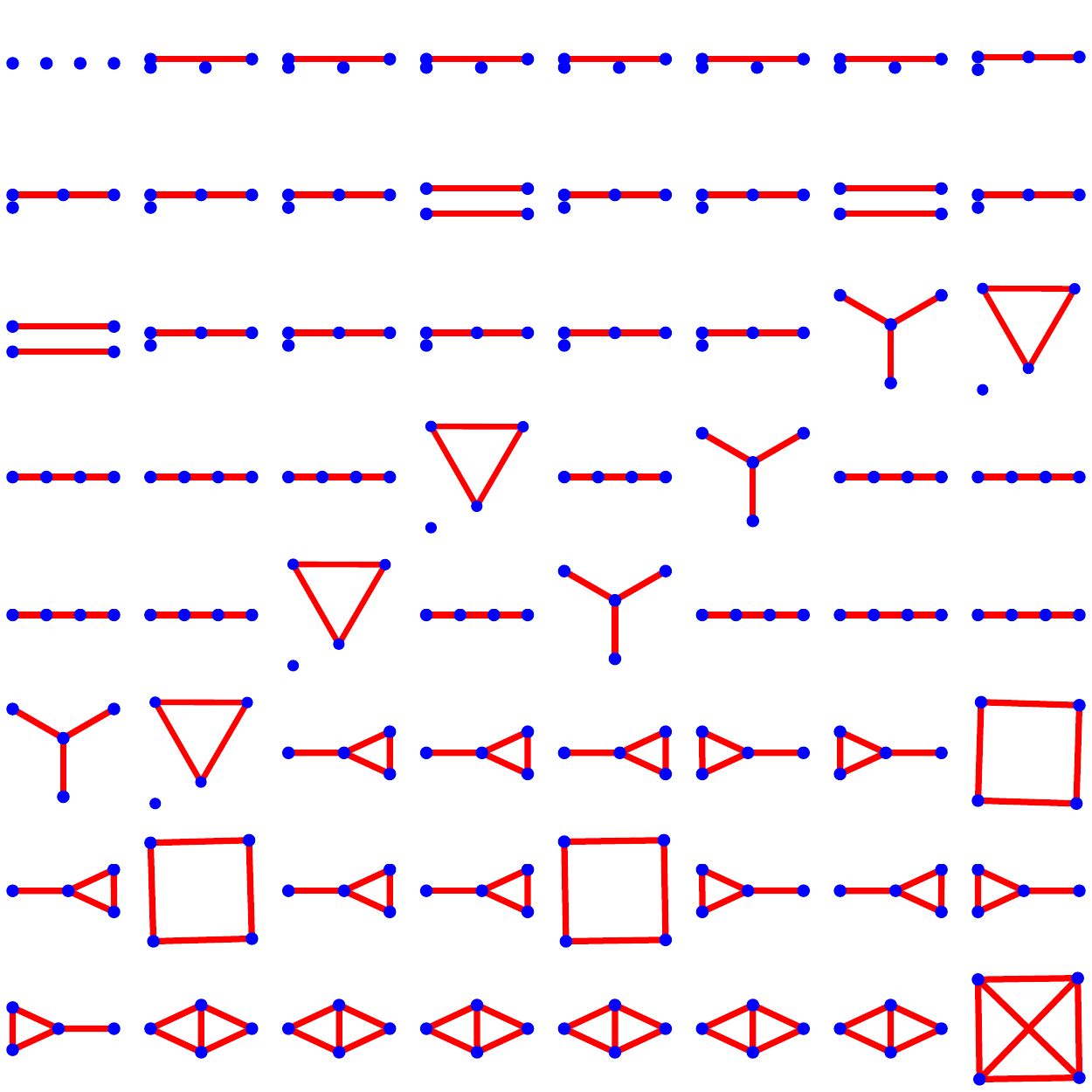}} }
\parbox{6.0cm}{ \scalebox{0.55}{\includegraphics{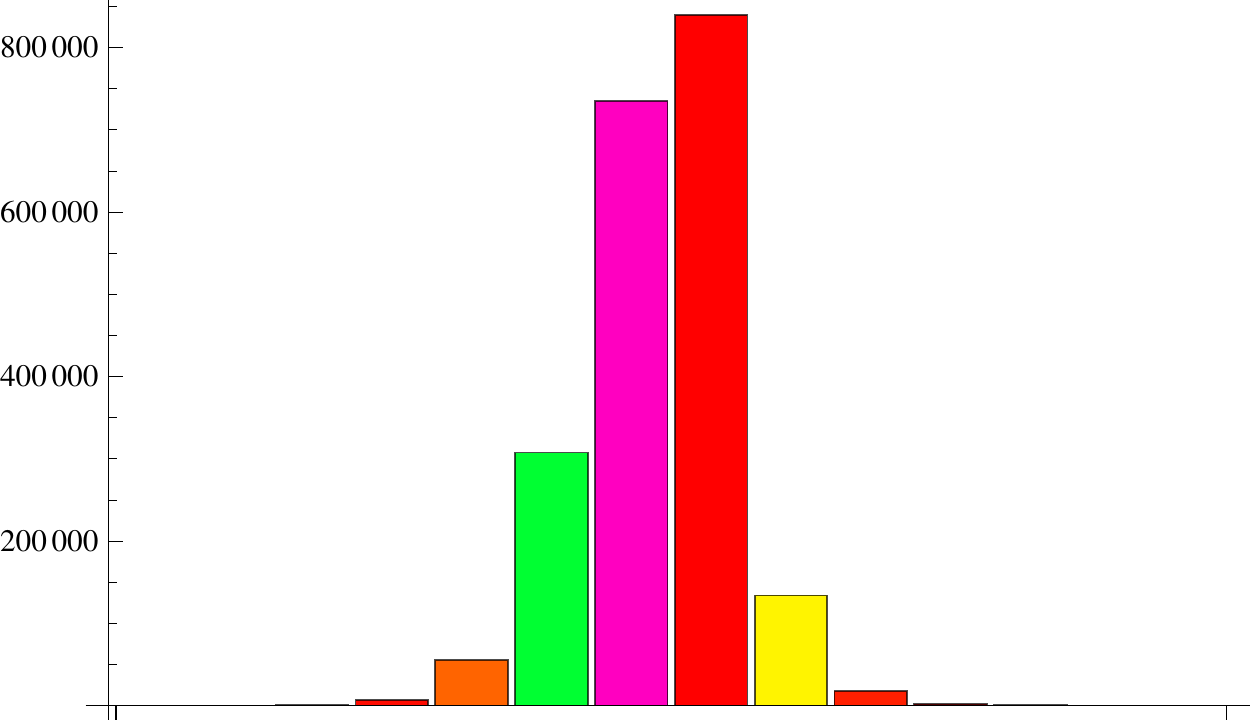}} }
\caption{
The total number of triangles $K_3$ in all subgraphs of $K_n$ is $2^{\B{n}{2}} \B{n}{3}/8$.
To the left, we see all the $2^{6}=64$ subgraphs of $K_4$. 
There are 32 triangles in total so that ${\rm E}[v_3]=1/2$. 
To the right we see the statistics of the Euler characteristic on 
the probability space $G(7,1/2)$. This is the last row of Table~\ref{eulertable}.
The most frequent case is $\chi=1$ followed by $\chi=0$. 
}
\label{eulercharacteristic}
\end{figure}

\begin{table}[h!]
\begin{tiny}
\begin{center}
\begin{tabular}{|llllllllllllll|} \hline
$\chi=$&$-5$ &$-4$ & $-3$ & $-2$ & $-1$ & $0$ & $1$& $2$ & $3$ & $4$ & $5$ & $6$ & $7$ \\ \hline \hline
n=1   &   &   &     0 &    0 &     0 &     0 &      1 &     0 &    0 &    0 &   0 &  0  & 0  \\ \hline
n=2   &   &   &     0 &    0 &     0 &     0 &      1 &     1 &    0 &    0 &   0 &  0  & 0  \\ \hline
n=3   &   &   &     0 &    0 &     0 &     0 &      4 &     3 &    1 &    0 &   0 &  0  & 0  \\ \hline
n=4   &   &   &     0 &    0 &     0 &     3 &     35 &    19 &    6 &    1 &   0 &  0  & 0  \\ \hline
n=5   &   &   &     0 &    0 &    10 &   162 &    571 &   215 &   55 &   10 &   1 &  0  & 0  \\ \hline 
n=6   &   &   &    10 &  105 &  1950 &  9315 &  16385 &  4082 &  780 &  125 &  15 &  1  & 0  \\ \hline
n=7   & 35&420&  6321 &54985 &307475 &734670 & 839910 &133693 &17206 & 2170 & 245 & 21  & 1  \\ \hline
\end{tabular}
\end{center}
\end{tiny}
\caption{The Euler characteristic statistics of 
subgraphs of $K_n$. We see that $\chi=1$ is the most frequent situation in all cases $n=1$ to $n=7$.}
\label{eulertable}
\end{table}

{\bf Remarks.} \\
{\bf 1.} The Euler characteristic expectation ${\rm E}_{n,p}[\chi]$ as a function of $n$ 
oscillates between different signs for $n \to \infty$ 
because for each fixed $k$, the function $n \to \B{n}{k}/2^{\B{n}{2}}$ dominates in some range
than is taken over by an other part of the sum.\\
{\bf 2.} We could look at values of $p_n$ for which the expectation value of the dimension gives 
${\rm E}_{n,p}[{\rm dim}]=2$ and then look at the limit of the Euler characteristic.  \\
{\bf 3.} The clustering coefficient which is $3 v_2/a_2$ where $a_2$ 
are the number of pairs of adjacent edges is a random variable studied already. It would
be interesting to see the relation of clustering coefficient with dimension. \\
{\bf 4.} The formula appears close to $\sum_{k=1}^n (-1)^k \B{n}{k} p^k$ which simplifies to $1-(1-p)^n$
and which is monotone in $p$. But changing the $p^k$ to $p^{\B{n}{k}}$ completely changes the function
because it becomes a Taylor series in $p$ which is sparse. We could take the sequence
${\rm E}_{n,p}[\chi]/\B{n}{n/2}$ to keep the functions bounded, but this converges to $0$ for $p<1$.
A natural question is whether for some $p(n)$, we can achieve that ${\rm E}_{n,p(n)}[\chi]$ converges
to a fixed prescribed Euler characteristic. Since for every $n$, we have ${\rm E}_{n,0}[\chi]=n$
and ${\rm E}_{n,0}[\chi] = 1$ and because of continuity with respect to $p$, we definitely can 
find such sequences for prescribed $\chi \geq 1$. \\
{\bf 5.} Despite correlation between the random variables $v_k$ and different expectation values,
its not impossible that the distribution of the random variable 
$X_n(G) = (\chi(G)-{\rm E}[\chi(G)])/\sigma(\chi(G))$ on $G(n,p)$ could converge weakly in the limit 
$n \to \infty$ if $\sigma(\chi)$ is the standard deviation. We would like therefore to find
moments of $\chi$ on $G(n,p)$. This looks doable since we can compute moments and correlations of 
the random variables $v_k$. The later are dependent:
the conditional expectation ${\rm E}[ v_3=1 | v_2=0 \; ]$ for example is
zero and so different from ${\rm E}[v_3=1] \cdot {\rm E}[v_2=0 \; ] >0$. \\
{\bf 6.} When measuring correlations and variance, we have no analytic formulas yet and for $n \geq 8$ we had 
to look at Monte Carlo experiments. Tests in cases where we have analytic knowledge and can compare the
analytic and experimental results indicate that for $n \leq 15$ and sample size $10'000$
the error is less than one percent. \\

\begin{figure}
\parbox{15cm}{
\parbox{7.0cm}{ \scalebox{0.22}{\includegraphics{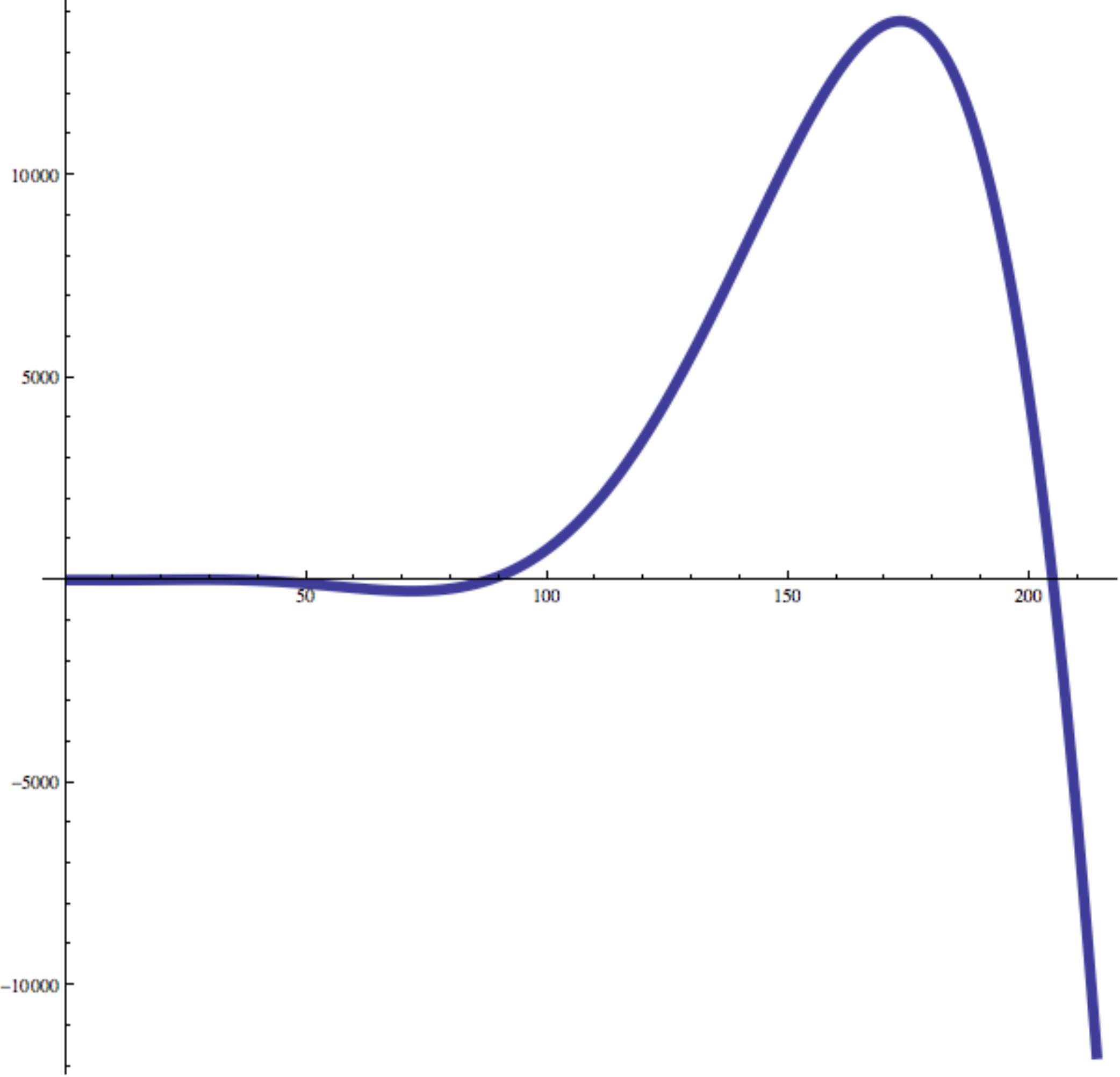}} }
\parbox{7.0cm}{ \scalebox{0.22}{\includegraphics{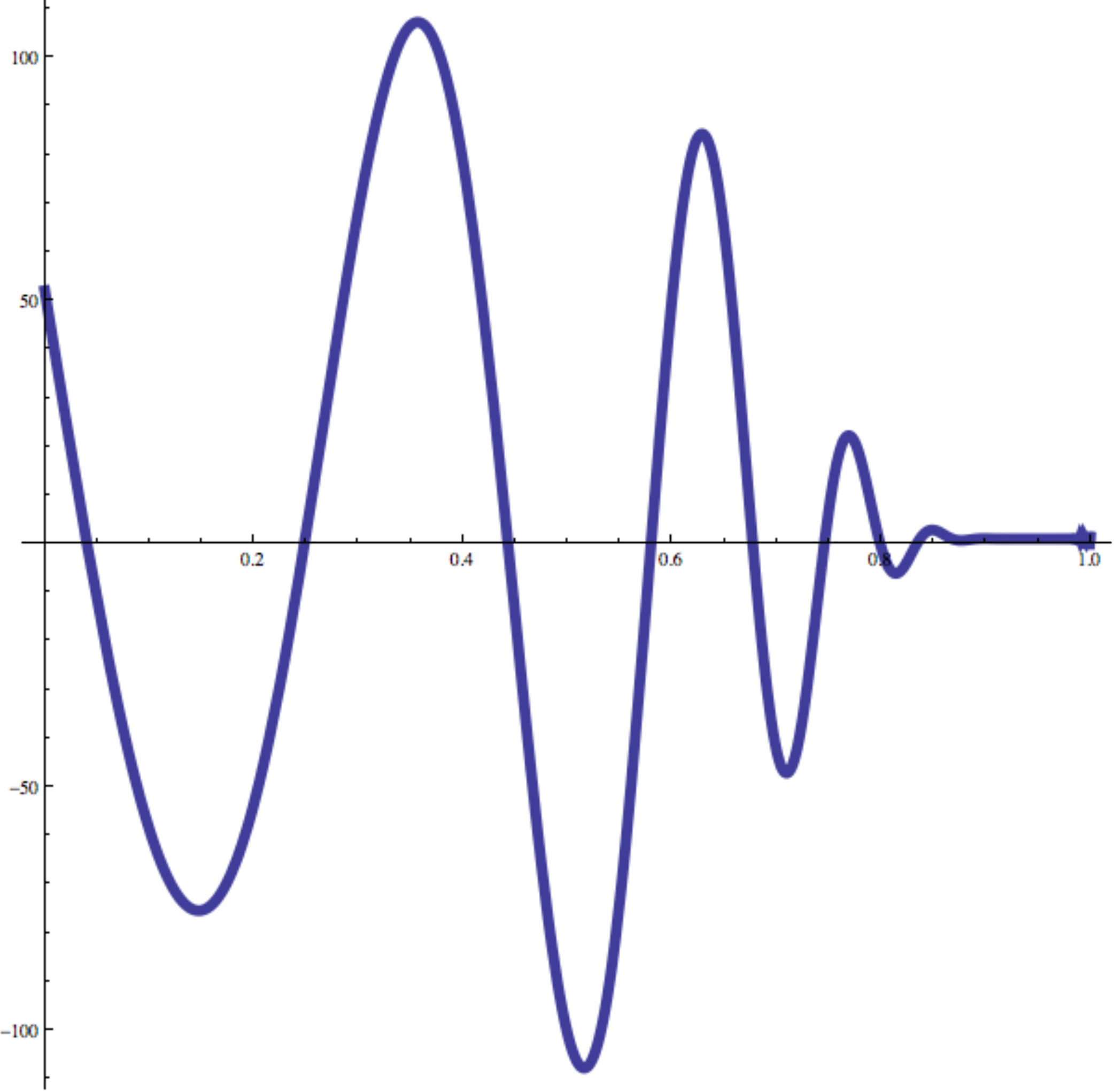}} }
}
\caption{
The average Euler characteristic as a function of $n$ oscillates. The intervals, where the behavior 
is monotone gets larger and larger however, the amplitudes grow too. Also the $p$ dependence is
interesting. The right picture shows the average Euler characteristic as a function of $p$ 
for a fixed order $n=52$. The function is $n-1$ at $p=0$ and $1$ at $p=1$. 
}
\label{eulerchar}
\end{figure}

\begin{table}[h!]
\begin{center}
\begin{tabular}{l} \hline
$d_1(p)=1$\\
$d_2(p)=2-p$\\
$d_3(p)=3 - 3p + p^3$\\
$d_4(p)=4 - 6p + 4p^3 - p^6$\\
$d_5(p)=5 - 10p + 10p^3 - 5p^6 + p^{10}$\\ \hline
\end{tabular}
\end{center}
\caption{Polynomials $d_n(p)$ which give the average Euler characteristic 
${\rm E}_p[\chi]$ on the probability space $G(n,p)$.}
\end{table}

\section{Statistical signatures}

For any graph $G=(V,E)$ called "host graph" we can look at the dimension and Euler characteristic signature functions 
$$   f(p) = {\rm E}_p[ {\rm dim}(H) \; ],  \hspace{1cm}  g(p) = {\rm E}_p[ \chi(H) \; ]  \;  $$
which give the expected dimension and Euler characteristic on the probability space $G(p)$ of all subgraphs of $G$
if every edge is turned on with probability $p$. These are
polynomials in $p$. We have explicit recursive formulas in the case of $K_n$, in which case the 
coefficients of $f$ are integers. We can explore it numerically for others. 
The signature functions are certainly the same for isomorphic graphs but they do not characterize the graph yet.
Two one dimensional graphs which are unions of disjoint cyclic graphs, 
we can have identical signature functions $f,g$ if their vertex cardinalities agree. The union 
$C_5 \cup C_5$ and $C_4 \cup C_6$ for example have the same signature functions $f,g$. \\

Since the global signature functions are not enough to characterize a graph, 
we can also look at the local dimension and Euler characteristic signature functions
$$   h_v(p) = {\rm E}_p[1+{\rm dim}(H,S(v)], \hspace{1cm}  k_v(p) = {\rm E}_p[ K(H,v) \; ] \; ] \; ,$$
where $H$ runs over all subgraphs and $v$ over all vertices. Here $K(H,v)$ is the curvature
of the vertex $v$ in the graph $H$. Of course, by definition of dimension and by the Gauss-Bonnet-Chern 
theorem for graphs. 
$$  \frac{1}{|V|} \sum_{v \in V} h_v(p) =  f(p), \hspace{1cm}  \sum_{v \in V} k_v(p) = g(p)  \; . $$
Since for one-dimensional graphs without boundary, the curvature is zero everywhere, 
also these do not form enough invariants for graphs. \\

We could also look at the dimension and curvature correlation matrices 
$$  A(p) = {\rm Corr}[h_v(p),h_w(p)], \hspace{1cm}  B(p) = {\rm Corr}[k_v(p),k_w(p)]  \;  $$
which are $n \times n$ matrix-valued functions on $[0,1]$ if $|V|=n$ is the cardinality of
the vertex set. We have 
$$ \frac{1}{|V|} {\rm tr}(A(p)) = {\rm Var}_p[ {\rm dim} ],  
   \hspace{1cm}  {\rm tr}(B(p)) = {\rm Var}_p[ \chi \; ] \; . $$

\begin{figure}
\parbox{6.0cm}{ \scalebox{0.22}{\includegraphics{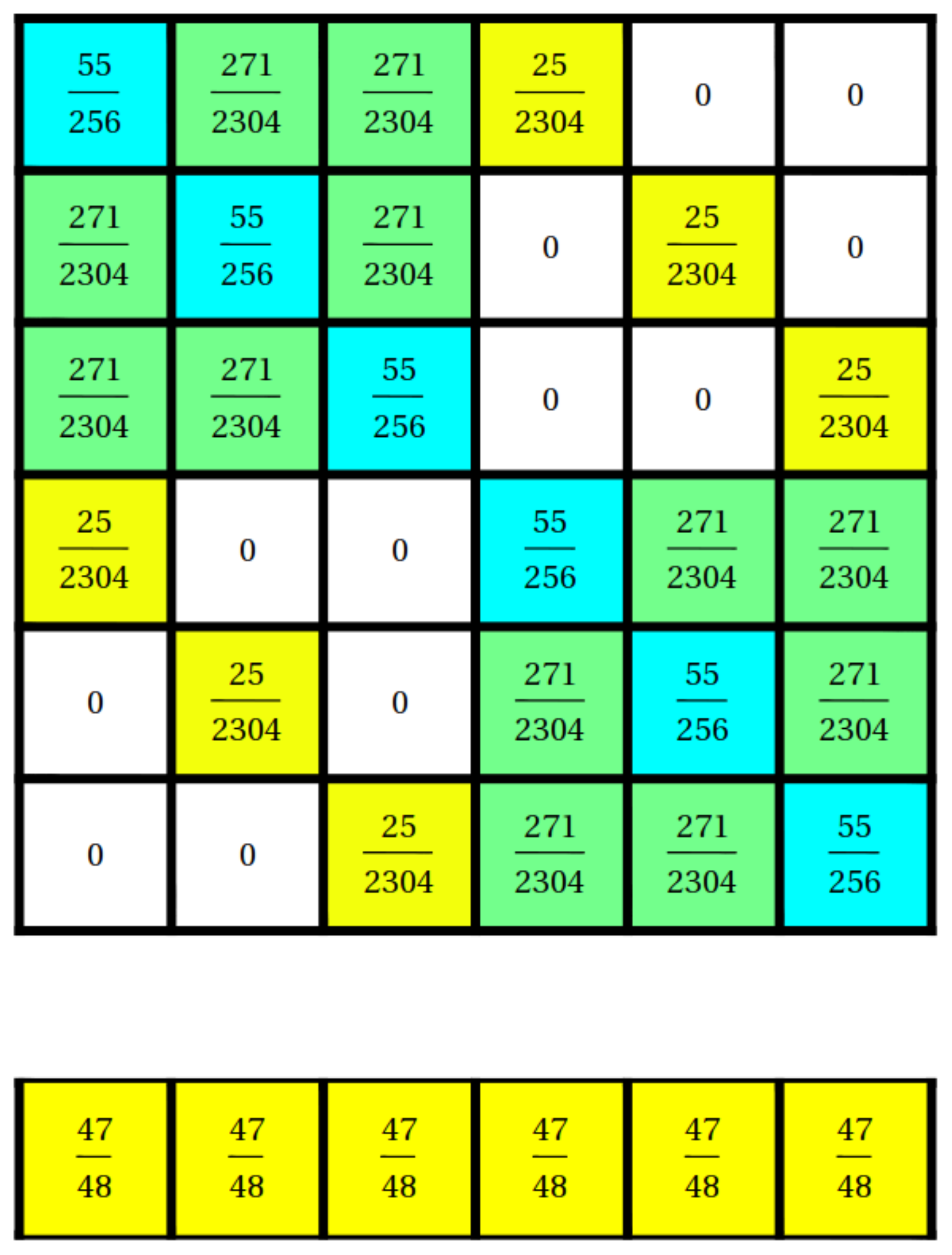}} }
\parbox{6.0cm}{ \scalebox{0.22}{\includegraphics{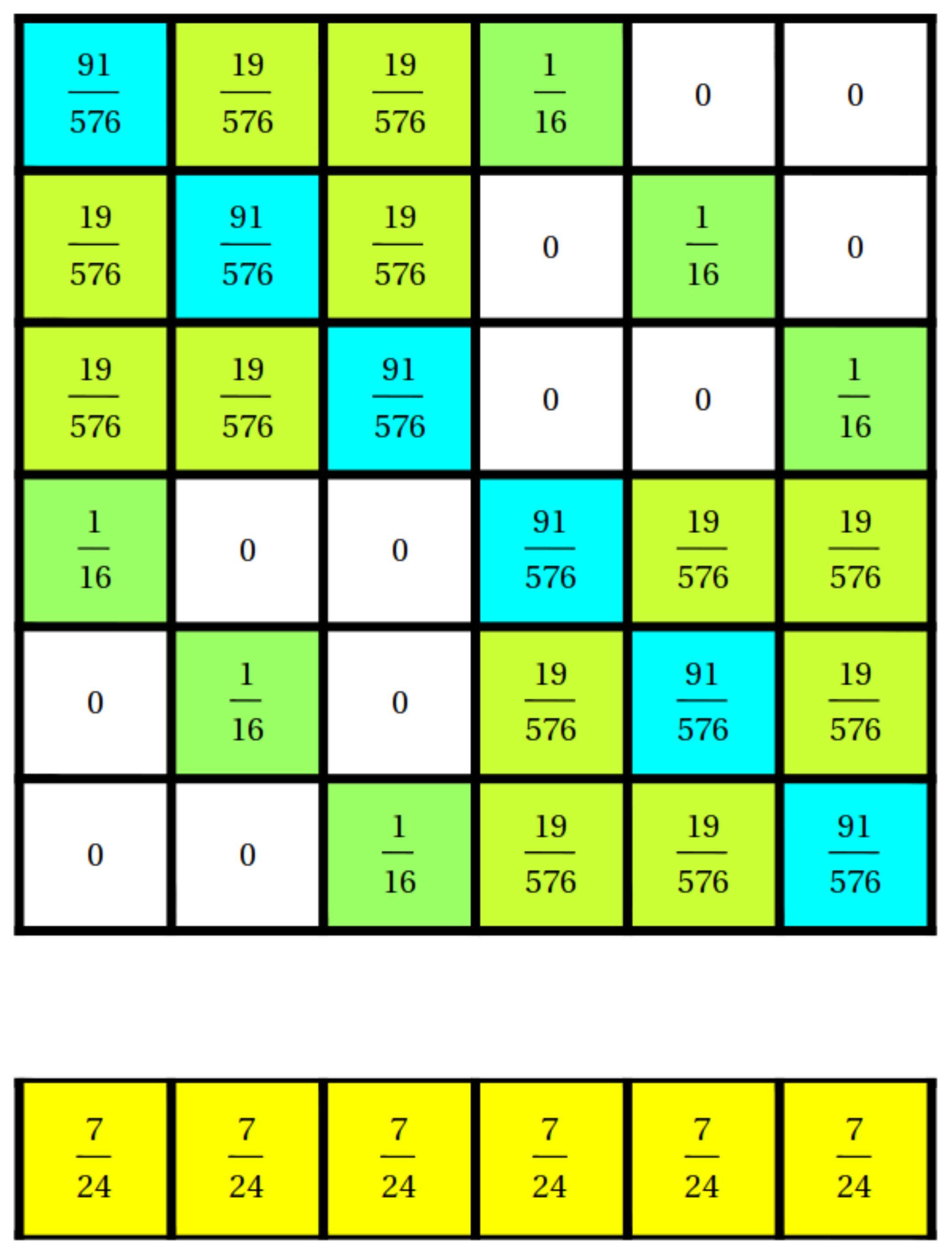}} }
\caption{
The dimension correlation matrix $A(1/2)$ and the 
curvature correlation matrix $B(1/2)$ for the Petersen graph $G=P(3,2)$ is shown.
Below the dimension correlation matrix, we see the average local dimension.
Below the curvature correlation matrix, we see the average curvature. 
The dimensions and correlations were computed in the probability space $G(p)$ of all subgraphs of $G$ 
which have the vertex set of $G$. The curvatures and dimensions at vertices which are not 
connected are uncorrelated. 
}
\label{petersenturan}
\end{figure}

\begin{figure}
\parbox{7.0cm}{ \scalebox{0.35}{\includegraphics{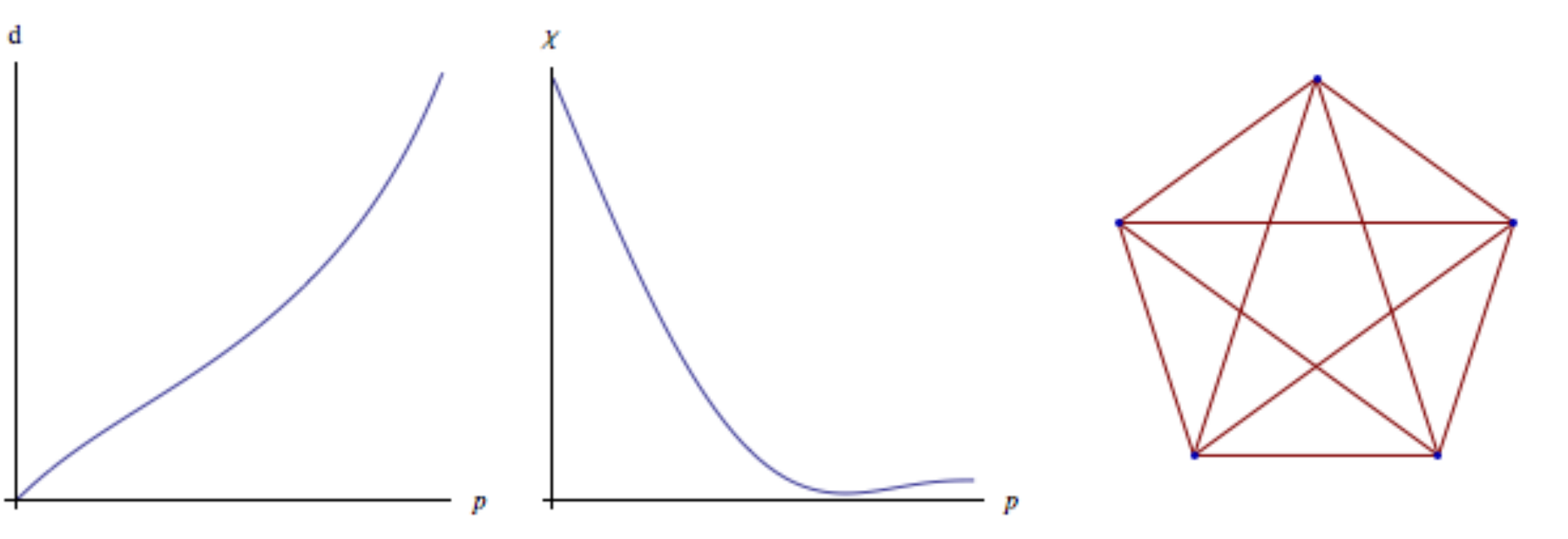}} }
\parbox{7.0cm}{ \scalebox{0.35}{\includegraphics{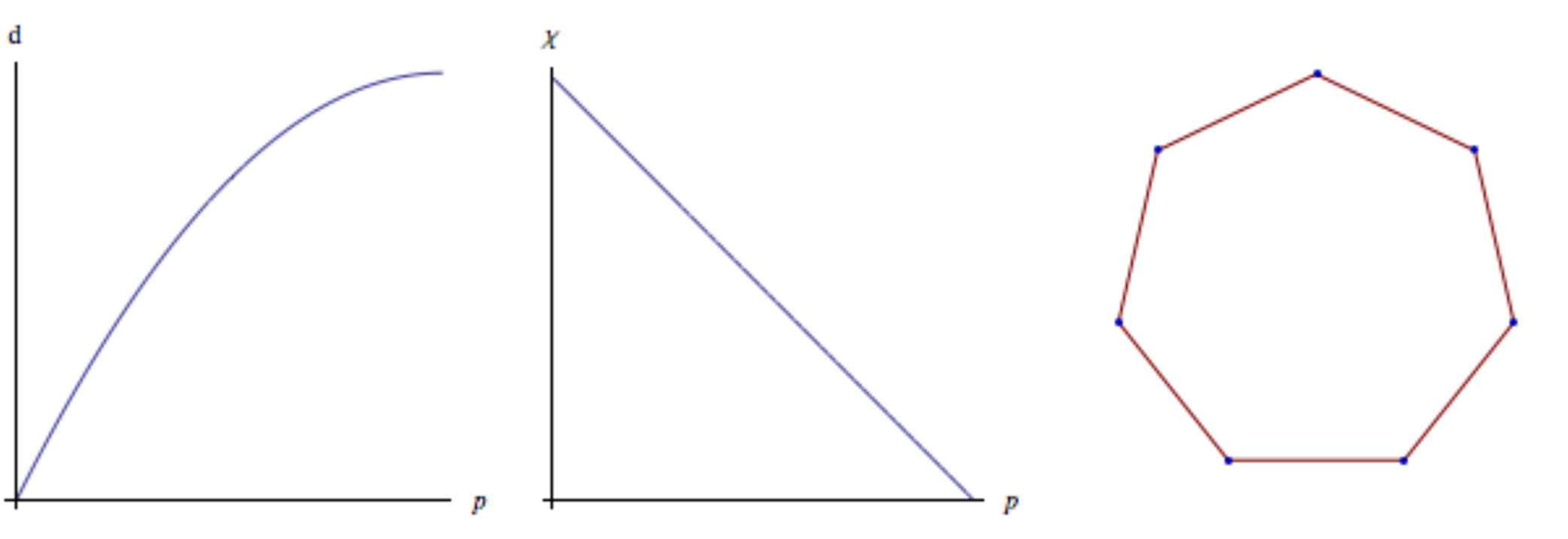}} }
\caption{
The dimension and Euler characteristic signature functions for the complete 
graph $K_5$ are
$$ f(p) = p (p^9-p^8-p^7+5 p^5-3 p^4-5 p^3+10 p^2-6 p+4) $$
$$ g(p) = 5 - 10p + 10p^3 - 5p^6 + p^{10}  \; . $$
The bottom figure shows the case when the host graph is the circular graph $C_7$, where 
$$  f(p) = p(2-p), \hspace{1cm}   g(p) = 7(1-p)  $$ 
The function $f$ is the same for every one dimensional graph, and in general 
$g(p) = |V|(1-p)$. }
\label{signature}
\end{figure}

\hspace{5mm}

\begin{figure}
\parbox{6.0cm}{ \scalebox{0.22}{\includegraphics{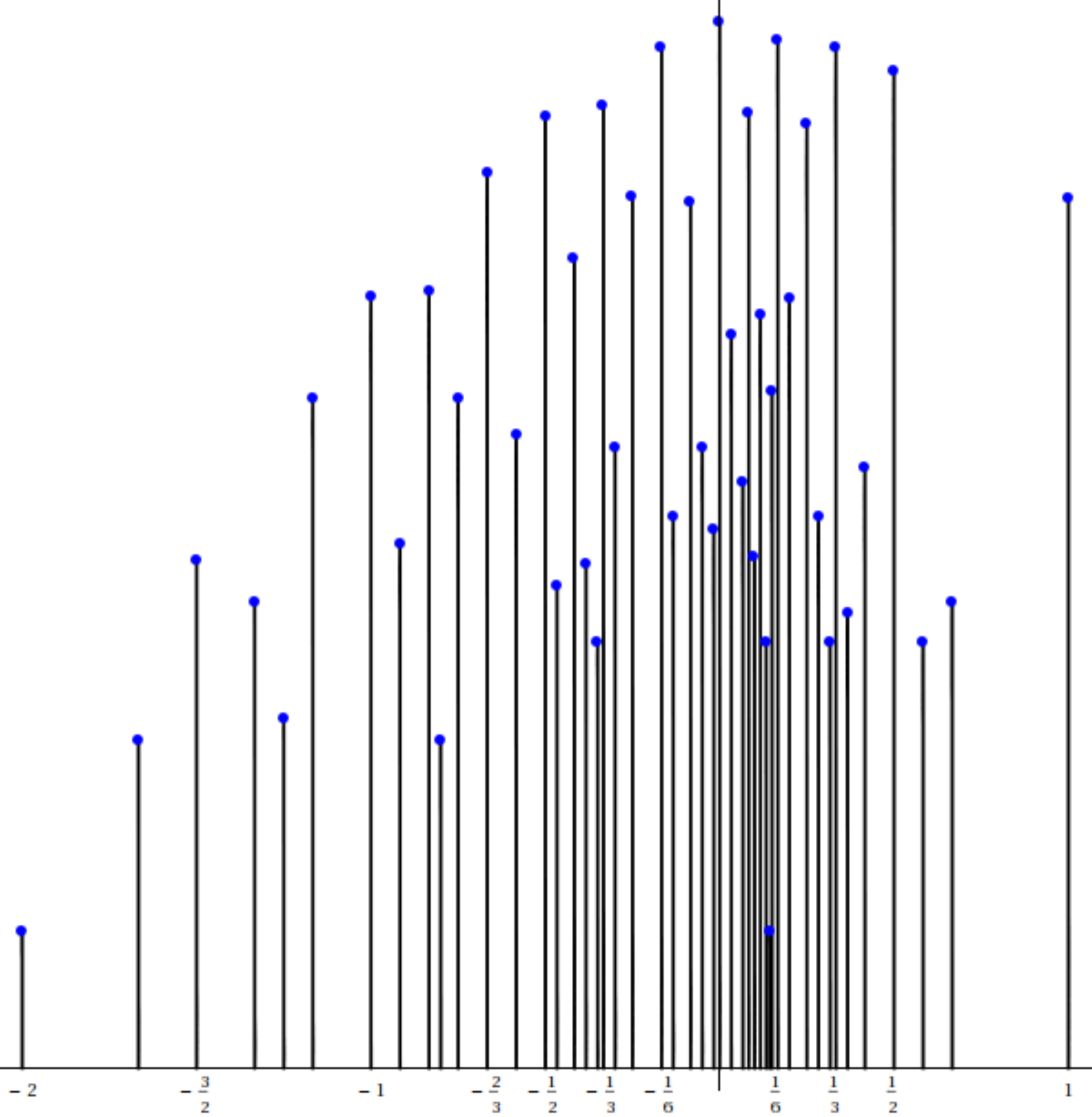}} }
\parbox{6.0cm}{ \scalebox{0.22}{\includegraphics{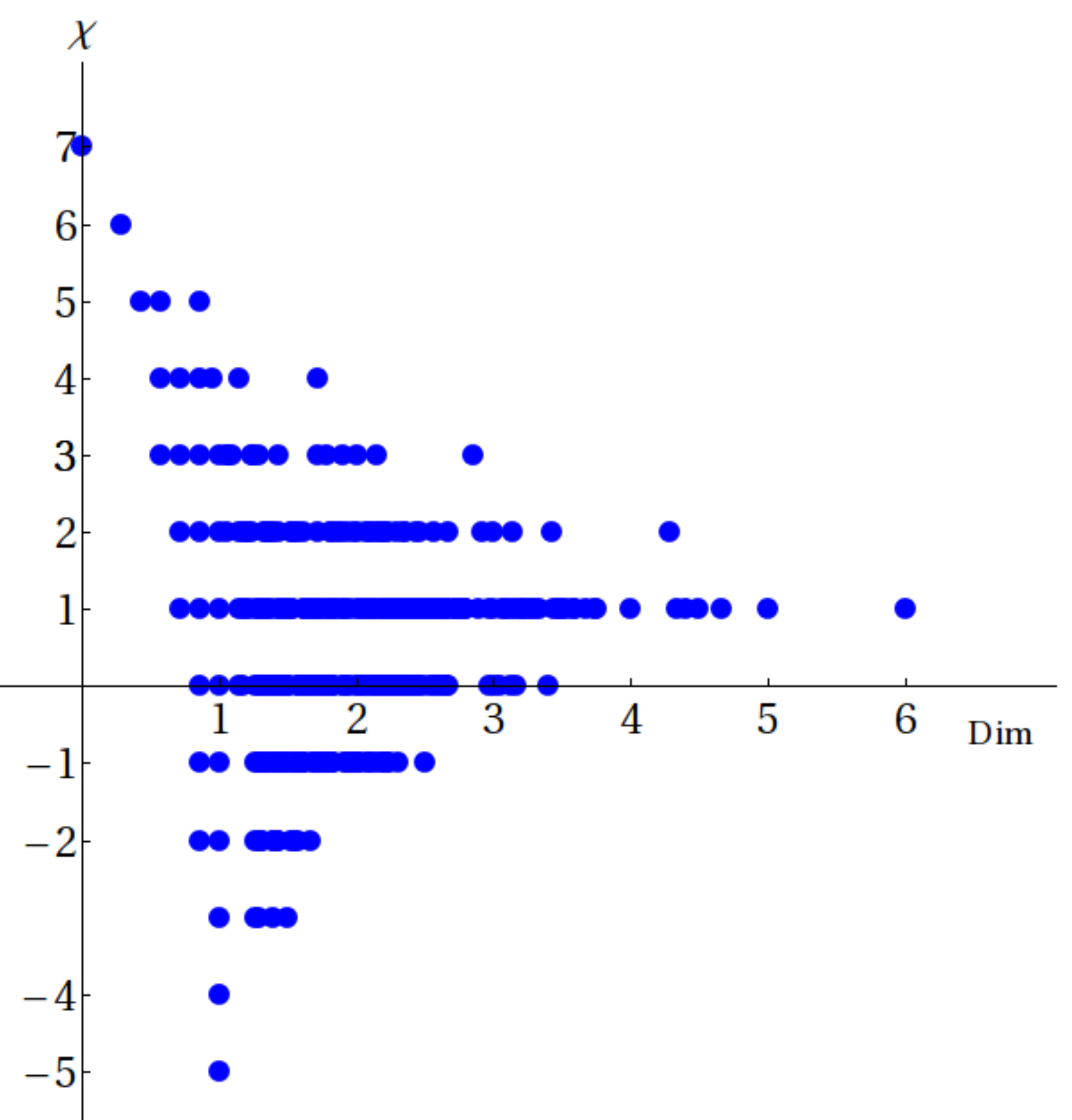}} }
\caption{
The left figure shows the logarithm of the number of vertices having a given curvature 
on the probability space $\{ V \times G(7,1/2)  \}$,
where $V=\{1,2,..,\; 7 \; \}$ is the vertex set on the set $G(7,1/2)$ of graphs with 
$7$ vertices. There are 47 different curvature values attained between $-2$ and $1$. 
The right figure shows the pairs $({\rm dim}(G),\chi(G))$,
where $G$ runs over all $2^{\B{7}{2}}$ graphs of $7$ vertices.
The Euler characteristic takes values between $-5$ and $7$, 
the dimension of the graph takes values between $0$ and $6$.
}
\label{dimchi}
\end{figure}

{\bf Remarks.} \\
{\bf 1.} In the computation to Figure~\ref{dimchi} we observed empirically that 
most nodes are flat: 2784810 of the $7 \cdot 2^{\B{7}{2}}$ vertices
have zero curvature. Graphs are flat for example if they are cyclic and one dimensional, which 
happens in $6!=720$ cases. The minimal curvature $-2$ is attained $7$ times
for star trees with central degree $6$.  \\
{\bf 2.} Extreme cases of in the ``dimension-Euler characteristic plane"
are the complete graph with $({\rm dim},\chi)=(6,1)$,
the discrete graph with $({\rm dim},\chi)=(0,7)$ as well as the
case $({\rm dim},\chi)=(1,-5)$ with minimal Euler characteristic.
An example is a graph with $12$ edges, $7$ vertices and no triangles
obtained by taking an octahedron, remove all 4 vertices in the xy-plane
then connect these 4 vertices with a newly added central point. \\

\begin{table}[h!]
\begin{center}
\begin{tabular}{|lllllll|} \hline
1&2             & 3            &        4        &   5                   &    6                      &  7      \\ \hline
 &              &              &                 &                       &                           &         \\ 
0&$\frac{-1}{4}$&$\frac{-19}{64}$&$\frac{-981}{4096}$&$\frac{-138043}{1048576}$&$\frac{-4889125}{1073741824}$&$\frac{540429585637}{4398046511104}$  \\
 &              &              &                 &                       &                           &         \\  \hline
\end{tabular}
\caption{
The correlations of ${\rm dim}$ and ${\rm \chi}$ for until $n=7$ for $p=1/2$. 
These are exact rational numbers because dimension is a rational number and Euler 
characteristic is an integer. The correlation has become positive the first time at $n=7$. 
For $n$ larger than $7$ we only numerically computed it with Monte Carlo runs in 
Figure~\ref{dimchi2}. 
}
\end{center}
\end{table}

\begin{figure}
\parbox{6.0cm}{ \scalebox{0.22}{\includegraphics{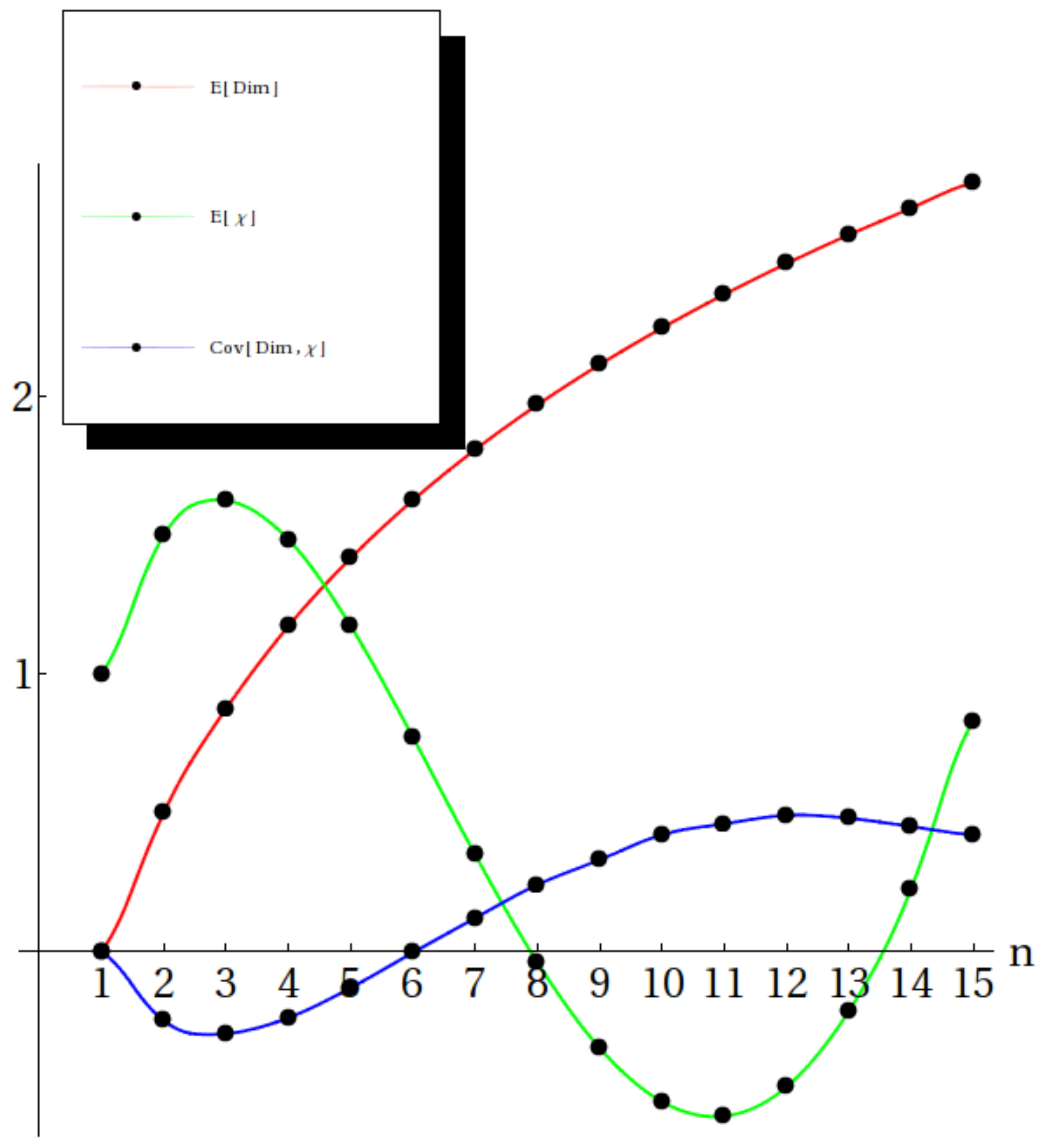}} }
\parbox{6.0cm}{ \scalebox{0.22}{\includegraphics{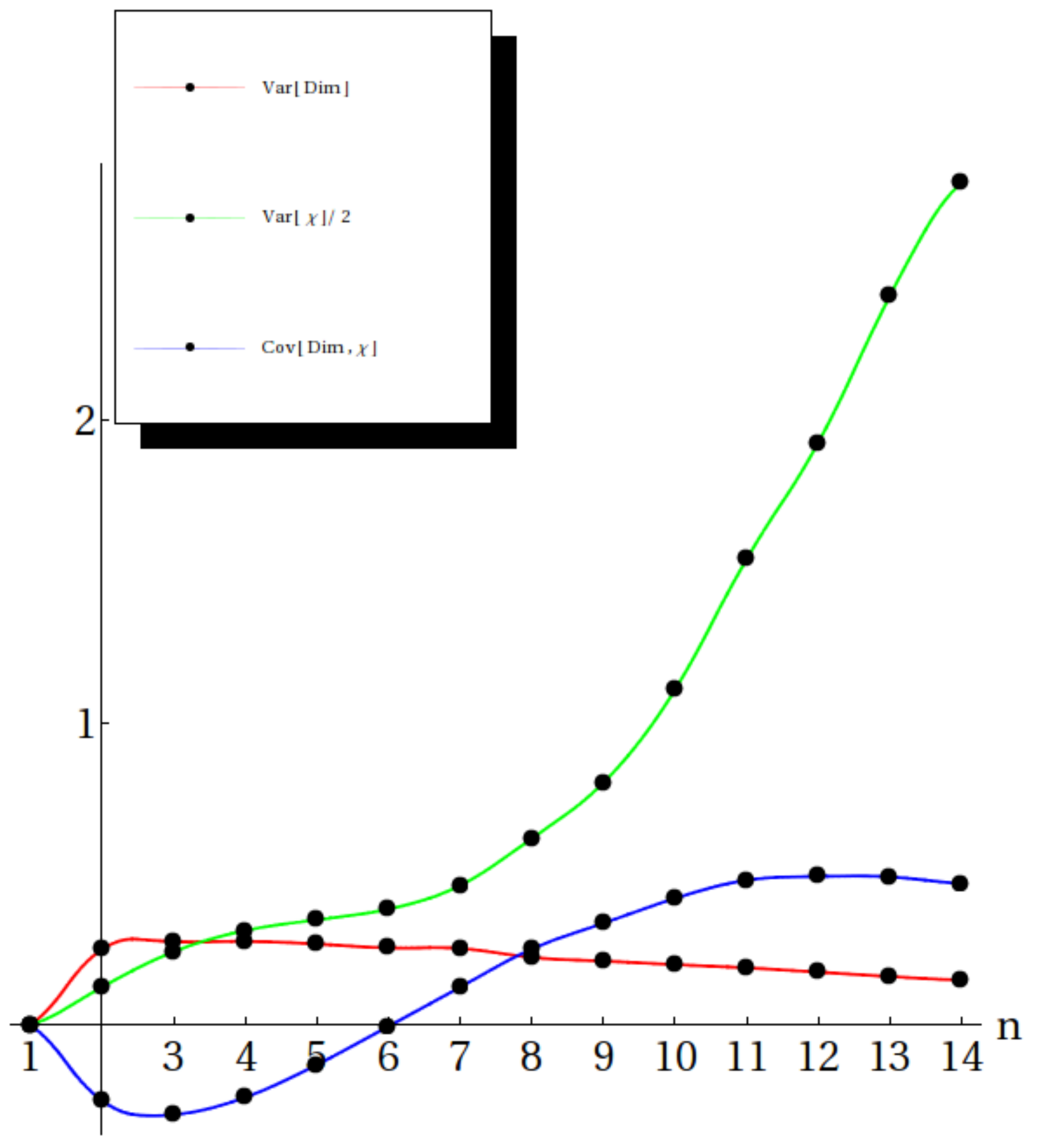}} }
\caption{
The left figure shows the expectation of the random variables ${\rm dim}$ and $\chi$
on $G(n,1/2)$ as well as ${\rm Cov}[{\rm dim},\chi]$ for $n=1,\dots,14$.
The right figure shows the variance ${\rm Var}[{\rm dim}], {\rm Var}[\chi]$
as well as ${\rm Cov}[{\rm dim},\chi]$ on $G(n,1/2)$ again for $n=1,\dots,14$. 
While the variance of the Euler characteristic grows rather rapidly on that interval,
the variance of dimension decreases on that interval. For $n=20$, we measure it to be 
about $0.1$. We can explore experimentally however only a very small part of the
probability space $G(20,1/2)$ consisting of $2^{\B{20}{2}} \sim 1.5 \cdot 10^{53}$ graphs.
}
\label{dimchi2}
\end{figure}

\begin{table}[h!]
\begin{center}
\begin{small}
\begin{tabular}{|lllllllllllll|} \hline
n   & 1 &   2 & 3 &    4 &    5 &    6 &    7 &    8 &    9 &   10 & 11  &  12 \\ \hline
min & 1 & 1/2 & 0 & -1/2 &   -1 & -1.5 & $-2$ & $ -2.5$ & $-3$ & $-3.5$ & $-4$  &  $-4.5$ \\ 
max & 1 &   1 & 1 &    1 &    1 &    1 & $ 1$ & $\geq   1$ &  $\geq    1$ & $\geq 4/3$ & $\geq 3/2$ &  $\geq 5/3$ \\  \hline
\end{tabular}
\end{small}
\end{center}
\caption{
The maximal and minimal curvature which can occur at points of graphs with $n$ vertices. 
This was settled for $n=1,\dots,7$ by checking over all graphs. The minimal curvatures are 
obtained at star shaped trees, where curvature satisfies $K(v) = 1-{\rm deg}(v)/2$.
For larger $n$, we ran Monte Carlo experiments over 10'000 random graphs. 
}
\end{table}

Having looked at the random variables ${\rm dim}$ and $\chi$ on $G(n,p)$ it might be of
interest to study the correlation 
$$ {\rm Corr}[{\rm dim},\chi] = {\rm E}[{\rm dim} \cdot \chi] - {\rm E}[{\rm dim}] \cdot {\rm E}[\chi] \;  $$
between them. The extremal cases of size $0$, order $n$ graphs of Euler characteristic $n$ and dimension $0$
or complete graphs with Euler characteristic $1$ and dimension $n-1$ suggest some anti correlation
between ${\rm dim}$ and $\chi$; but there is no reason, why there should be any correlation trend between 
dimension and Euler characteristic in the limit $n \to \infty$. Like Euler characteristic, it could oscillate. \\

While we have a feel for dimension of a large network as a measure of recursive "connectivity degree", 
Euler characteristic does not have interpretations except in geometric situations with definite constant 
dimensions. For example, for a two dimensional network, it measures the number of components minus the number of "holes",
components of the boundary $\delta G$ the set of points for which the unit sphere is one dimensional
but not a circle. While for geometric $d$ dimensional graphs, 
$\chi(G)$ has an interpretation in terms of Betti numbers, for general networks, both dimensions 
and curvatures varies from point to point and the meaning of Euler characteristic remains an enigma for 
complex networks.

\bibliographystyle{plain}

\end{document}